\documentclass[11pt, oneside]{article}   	
\usepackage{amsmath, amsthm, amsfonts, amssymb}
\usepackage{graphicx}
\usepackage{subcaption}
\usepackage[colorlinks,citecolor=blue,urlcolor=blue]{hyperref}
\usepackage{mdwlist}
\usepackage{relsize}
\usepackage{enumerate}
\newcommand\Ex{{\mathbb E}}

\newcommand\Prob{{\mathbb P}}

\newcommand\cA{{\mathcal A}}

\newcommand\cC{{\mathcal C}}

\newcommand\cK{{\mathcal K}}
\newcommand\cM{{\mathcal M}}

\newcommand\cP{{\mathcal P}}

\newcommand\cY{{\mathcal Y}}
\newcommand\cZ{{\mathcal Z}}

\newcommand\cQ{{\mathcal Q}}

\newcommand\N{{\mathbb N}}
\newcommand\Z{{\mathbb Z}}
\newcommand\R{{\mathbb R}}

\newcommand\bH{{\mathbb H}}

\newcommand\fX{{\mathfrak X}}

\newcommand\x{{\mathbf x}}
\newcommand\y{{\mathbf y}}
\newcommand\z{{\mathbf z}}

\newcommand\A{{\mathbf A}}
\newcommand\B{{\mathbf B}}
\newcommand\J{{\mathbf J}}
\newcommand\C{{\mathbf C}}

\newcommand\vto{\overset{v}{\to}}

\newcommand\dist{\overset{d}{=}}
\newcommand\one{{\bf 1}}

\DeclareMathOperator{\supp}{supp}

\newtheorem{theorem}{Theorem}[section]
\newtheorem{corollary}[theorem]{Corollary}
\newtheorem{lemma}[theorem]{Lemma}
\newtheorem{proposition}[theorem]{Proposition}

\theoremstyle{definition}
\newtheorem{definition}[theorem]{Definition}

\theoremstyle{remark}
\newtheorem{remark}[theorem]{Remark}

\providecommand{\keywords}[1]{{\textit{Keywords:}} #1}
\begin{document}
\title {Strong law of large numbers for Betti numbers in the thermodynamic regime}

\author{Akshay Goel \and Khanh Duy Trinh \and Kenkichi Tsunoda }
\newcommand{\Addresses}{{
  \bigskip
  \footnotesize

  A.G., \textsc{Faculty of Mathematics, Kyushu University, Japan}\par\nopagebreak
  \textit{E-mail address}: \texttt{a-goel@math.kyushu-u.ac.jp}

  \medskip

  K.D.T., \textsc{Research Alliance Center for Mathematical Sciences, Tohoku University, Japan}\par\nopagebreak
  \textit{E-mail address}: \texttt{trinh.khanh.duy.a3@tohoku.ac.jp}

  \medskip

    K.T., \textsc{Department of Mathematics, Graduate School of Science, Osaka University, Japan}\par\nopagebreak
  \textit{E-mail address}: \texttt{k-tsunoda@math.sci.osaka-u.ac.jp}
}}
\maketitle{}

%
%
%
%
\begin{abstract}
We establish the strong law of large numbers for Betti numbers of random \v{C}ech complexes built on $\mathbb R^N$-valued binomial point processes and related Poisson point processes in the thermodynamic regime. Here we consider both the case where the underlying distribution of the point processes is absolutely continuous with respect to the Lebesgue measure on $\mathbb R^N$ and the case where it is supported on a $C^1$ compact manifold of dimension strictly less than $N$. The strong law is proved under very mild assumption which only requires that the common probability density function belongs to $L^p$ spaces, for all $1\leq p < \infty$. 
\end{abstract}

\noindent\keywords{Betti numbers, random geometric complexes, thermodynamic regime, strong law of large numbers, manifolds}


\section{Introduction}
The emerging research area known as random topology comprises theoretical results that characterize the asymptotic behavior of topological properties of random objects \cite{BH, bob, kahle, kahle_lim, yogi, ysa}.
In addition to the mathematical value, such results also find many applications in manifold learning and topological data analysis as they provide tools for interpreting complex high dimensional data sets (see e.g.~\cite{carl, chen, ghrist}). One aspect of this area is the study of random geometric complexes and their topological properties called Betti numbers. Random geometric complexes, regarded as higher-dimensional generalizations of random geometric graphs, are generated from random points under certain deterministic rules. In this paper, we concentrate on random \v{C}ech complexes, a typical type of random geometric complexes, with the aim to establish the strong law of large numbers for their Betti numbers in the thermodynamic regime. The approach here, however, is general enough to apply to other types of geometric complexes.

For a finite set of points $\fX= \{x_1, x_2, \ldots, x_n\}$ in $\R^N$ and a radius $r>0$, the \textit{\v{C}ech complex}, denoted by $\cC(\fX, r)$, is defined to be an abstract simplicial complex consisting of all non-empty subsets $\sigma$ of $\fX$ for which $ \cap_{x \in \sigma} B(x, r) \neq \emptyset$. Here,  $B(x, r) = \{y \in\R^N\colon \|x-y\| \leq r\}$ denotes the closed ball of radius $r$ centered at $x$ with respect to the Euclidean norm $\left\|\cdot\right\|$. In other words, the \v{C}ech complex $\cC(\fX, r)$ is the nerve of the union of balls $\cup_{x \in \fX} B(x, r)$. As a consequence of the nerve lemma \cite{nerve_lemma}, the two objects are homotopy equivalent, and hence, have the same Betti numbers. Intuitively, the $k$th Betti number $\beta_k(X)$ of a topological space $X$ counts the number of $k$-dimensional `non-trivial cycles' or `holes' in it except $\beta_0(X)$, 
which counts the number of connected components of 
$X$. For instance, a $2$-dimensional sphere has $\beta_0 = \beta_2 = 1$ and $\beta_1 = 0$. Mathematically, the $k$th Betti number 
$\beta_k(X)$ is the rank of the $k$th homology group of $X$, which is defined algebraically using group theory. The knowledge of homology 
 groups is not needed to understand our results since we consider Betti numbers as the rank of homology groups with coefficients from some underlying field, which 
 can be defined easily by using elementary linear algebra. The reader interested in homology theory may refer to \cite{hatch, Munkres-1984} for a comprehensive introduction.

\begin{figure}
\centering
\begin{subfigure}[b]{0.3\textwidth}
\fbox{\includegraphics[width=\textwidth, height=3.7cm]{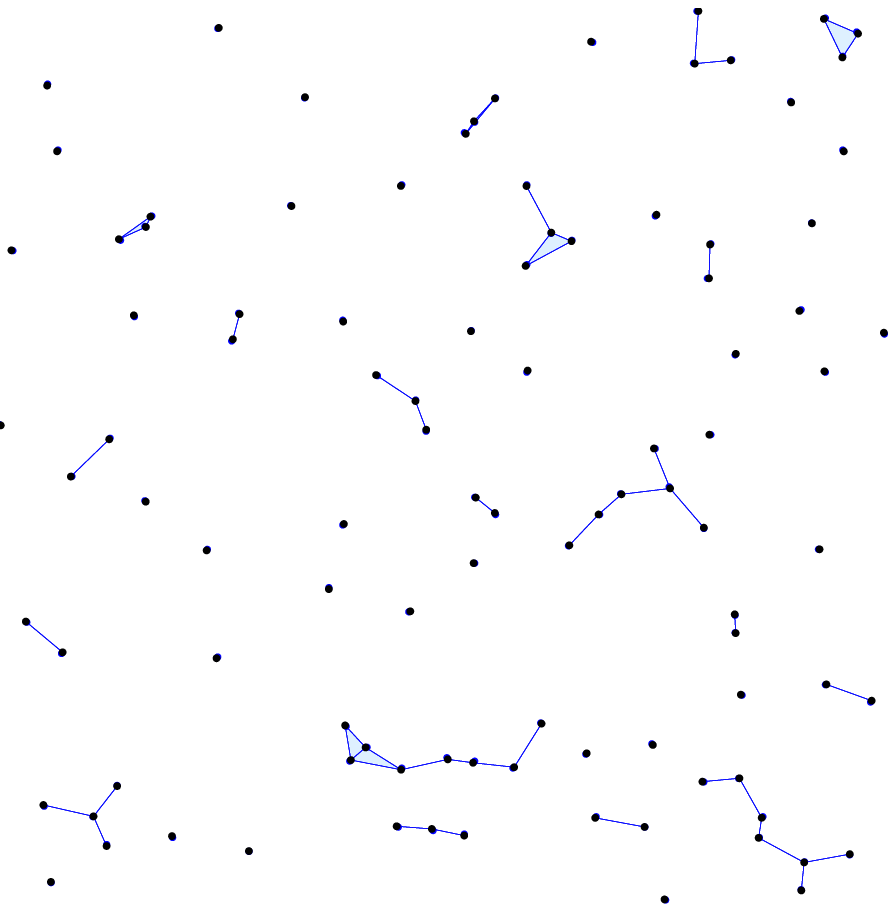}}
\caption{Sparse regime}
\end{subfigure}
\quad
\begin{subfigure}[b]{0.3\textwidth}
\fbox{\includegraphics[width=\textwidth, height=3.7cm]{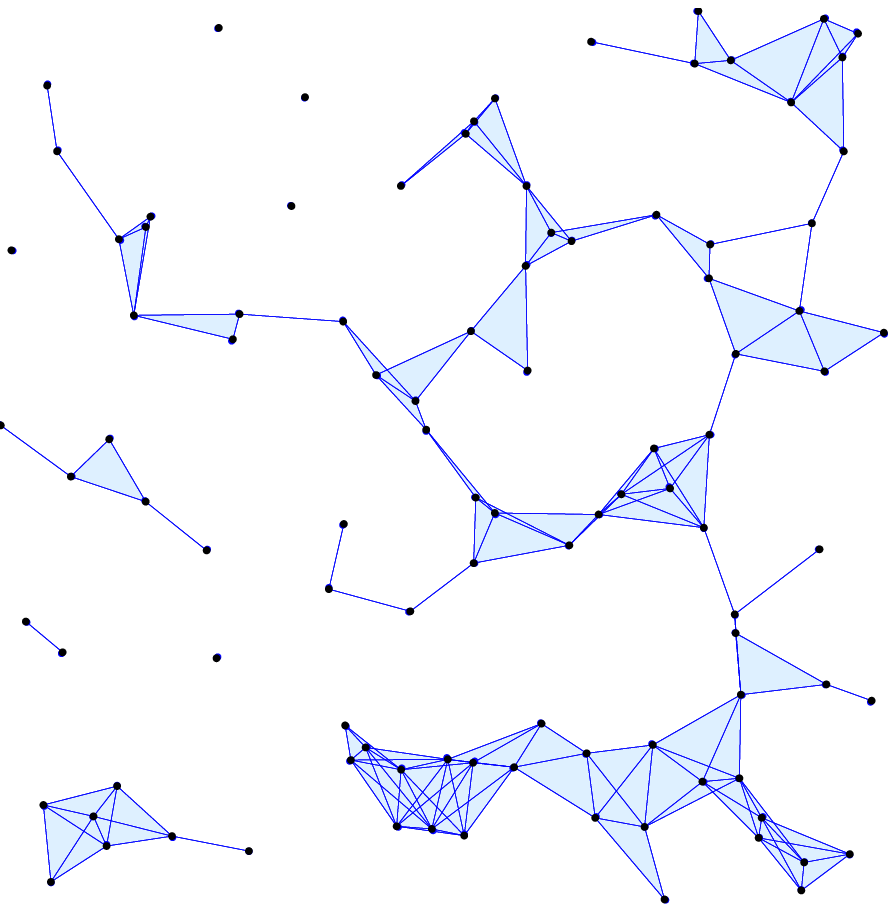}}
\caption{Thermodynamic regime}
\end{subfigure}
\quad
\begin{subfigure}[b]{0.3\textwidth}
\fbox{\includegraphics[width=\textwidth, height=3.7cm]{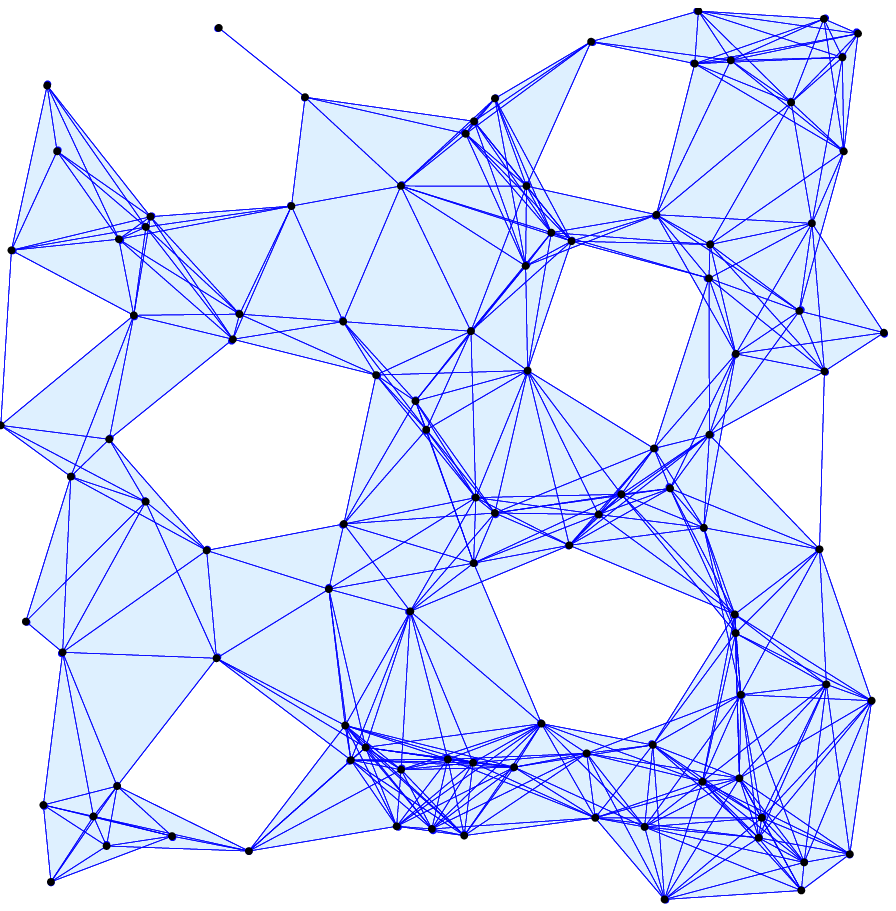}}
\caption{Dense regime}
\end{subfigure}
\caption{Illustration of different limiting regimes by constructing \v{C}ech complexes $\cC(\fX_{100}, r_{100})$ with $r_{100} = 0.03$ in (a), $0.06$ in (b) and $0.1$ in (c), where $\fX_{100}$ is a set of $100$ points drawn uniformly from $[0,1] \times [0,1] \subset \R^2$.}
\label{fig:regimes}
\end{figure}

Random \v{C}ech complexes $\cC(\fX, r)$ are built on random points $\fX$ with non-random radius $r > 0$. The source of random points may come from a stationary point process on $\R^N$, or an i.i.d.~(independent identically distributed) sequence $\{X_i\}_{i \ge 1}$ of $\R^N$-valued random variables. The paper will focus on the latter case and consider binomial point processes and Poisson point processes, which are defined to be the union of the first $n$ points $\fX_n=\{X_1, \dots, X_n\} $ and the first random $N_n$ points $\cP_n = \{X_1, \dots , X_{N_n}\}$, respectively. Here, $N_n$ has Poisson distribution with parameter $n$ and is independent of $\{X_i\}$. Let us first consider the case where the common distribution of $X_i$ has a probability density function $f(x)$ with respect to the Lebesgue measure on $\R^N$. We shall refer to this case as the Euclidean setting. It is known that there are three main regimes: sparse regime, thermodynamic regime and dense regime (see Figure~\ref{fig:regimes}) in which the limiting behavior of Betti numbers $\beta_k(\cC(\fX_n, r_n))$ and $\beta_k(\cC(\cP_n, r_n))$ is totally different. Here, $\{r_n\}$ is a non-random sequence of positive numbers tending to zero for which three regimes are divided according to the limit of $\{n^{1/N} r_n\}$: zero, finite or infinite. Note that the tools used to determine the limiting behavior in each regime are also different. As mentioned before, the paper concerns with  the thermodynamic regime, the regime where $n^{1/N}r_n \to r \in (0, \infty)$, in which basic problems such as laws of large numbers and central limit theorem have not been completely understood yet.

Let us introduce some known results before stating our main results. We begin with a result on homogeneous Poisson point processes. Denote by $\cP_L(\lambda)$ the restriction on $((-L/2)^{1/N}, (L/2)^{1/N}]^N$ of a homogeneous Poisson point process on $\R^N$ with intensity $\lambda \ge 0$. Then for $0 \leq k \leq N-1$, as $L \to \infty$ \cite[Theorem~3.5]{ysa},
\[
        \frac{\beta_k(\cC(\cP_L(\lambda), r))}{L} \to \hat \beta_k^{(N)}(\lambda, r) ~\text{\rm{a.s.}},
\]
where a.s.~stands for the almost sure convergence. Note that as a consequence of the nerve lemma, $ \beta_k(\cC(\fX, r)) \equiv 0$, if $k \geq N$, for any finite set $\fX \subset \R^N$ and any $r \geq 0$. Thus, we set $\hat\beta_k^{(N)} (\lambda, r) = 0$, if $k \geq N$. Now in the Euclidean setting, the following strong law of large numbers for $\beta_k(\cC(\fX_n, r_n))$ and $\beta_k(\cC(\cP_n, r_n))$ in the thermodynamic regime holds, i.e., as $n \to \infty$ with $n^{1/N}r_n  \to r \in (0, \infty)$, 
\[
	\frac{\beta_k(\cC(\fX_n, r_n))}{n} \left(\text{resp.~} \frac{\beta_k(\cC(\cP_n, r_n))}{n}\right)\to \int_{\R^N} \hat \beta_k^{(N)}(f(x), r) dx ~\text{\rm{a.s.}},
\]
provided that the probability density function $f(x)$ is Riemann integrable, has convex compact support and is bounded both below and above on the support \cite{Duy-2016, ysa}. Although in stochastic geometry, weak and strong laws of large numbers have been established for a general class of local functionals \cite{pen2007law, pen2003weak}, Betti numbers do not belong to that class. Thus, the study of Betti numbers needs further development. We propose here an elementary approach to show the strong law of large numbers, which can remove all the above technical conditions on $f(x)$ (see Theroem~\ref{thm:euclidb} after removing the symbol $\rho$ and taking $D(x) \equiv1$), and can apply to the problem on manifolds as well.

Now let $\cM \subset \R^N$ be a $C^1$ compact manifold of dimension $m < N$. Assume that the underlying distribution is supported on $\cM$ and has a probability density function $\kappa(z)$ with respect to the volume form $dz$ on $\cM$. It means that if $Z$ is a $\R^N$-valued random variable having density $\kappa$ then for every $A \subset \R^N$, 
\[
           \Prob(Z \in A) = \int_{A \cap \cM} \kappa(z) dz,
\]
where $dz$ is the volume form on $\cM$.
We shall refer to this case as the manifold setting. To distinguish from the Euclidean setting, we denote binomial point processes and Poisson point processes on manifold by $\cZ_n$ and $\cQ_n$, respectively.  It has been shown that results in the Euclidean setting can be naturally extended to the manifold setting. Here, the three regimes are divided according to the limit of $\{n^{1/m}r_n\}$. However, in the thermodynamic regime, the only existing result for Betti numbers is the linear growth of their expected value {\cite[Theorem~4.3]{bob}}. 
We improve this result by showing the strong law of large numbers stated below. 
In conclusion, we completely establish the strong law of large numbers in both the Euclidean setting and the manifold setting in this paper. The question on central limit theorem in the Euclidean setting is partially answered in \cite{owada2018, Trinh-2018} under a technical condition that the limiting radius $r$ is small enough. So we can say central limit theorems in both the settings are still open.
\begin{theorem}[For Manifolds]\label{thm:manifoldb}
Assume that the common probability density function 
$\kappa (z)$ is supported on an $m$-dimensional compact $C^1$ manifold $\cM \subset \R^N$ and for all $j \in \N$, $\int_{\cM} \kappa(z)^j dz < +\infty$. Then 
as $n \to \infty$ with $n^{1/m}r_n \to r \in (0, \infty) $, 
\[
 \frac{\beta_k(\cC(\cZ_n, r_n))}{n} \left(\text{resp.~}\frac{\beta_k(\cC(\cQ_n, r_n))}{n} \right) \to \int_{\cM} \hat \beta_k^{(m)} (\kappa(z), r) dz ~\text{\rm{a.s.}}
 \]
 Here, $\hat \beta_k^{(m)} (\lambda, r)$ is the limit of Betti numbers in case of homogeneous Poisson point processes on $\R^m$ (not on $\R^N$).
\end{theorem}
Note that it may be possible that for some $m \leq k < N$, $\beta_k(\cZ_n, r_n) > 0$, but  
$\hat\beta_k^{(m)} (\lambda, r) \equiv 0$ for all $k \geq m$. Note also that the study of the zeroth Betti number, which coincides with the number of connected components in a geometric graph, has a rich literature. A formula for $\hat \beta_0^{(m)}(\lambda, r)$ can be found in \cite[Theorem~13.25]{Penrose-book}. In some sense, the case $m = 2$ is completely understood because $\hat \beta_1^{(2)}(\lambda, r)$ can be deduced from the limiting behavior of the Euler characteristic \cite{bob}. These are all the cases where the explicit formula for $\beta_k^{(m)}(\lambda, r)$ has been known. Now for general $m$, and for $0 \le k \le m-1$, we gather here some known properties of $\hat \beta_k^{(m)}(\lambda, r)$.
\begin{enumerate}[(i)]
\item \textit{Scaling property:} For any $\theta > 0$,  $\hat \beta_k^{(m)}(\lambda, r) = \frac{1}{\theta} \hat \beta_k^{(m)}\left(\lambda \theta, \frac{r}{\theta^{1/m}}\right)$. 
 
\item \textit{Continuity and positivity:} $\hat \beta_k^{(m)}(\lambda, r)$ is a continuous function in both $\lambda$ and $r$, which is positive if $\lambda, r > 0$.

\item \textit{Exponential decay:} Let $\omega_m$ be the volume of a unit ball in $\R^m$ and $r \in (0, \infty)$. Then
\[
\hat \beta_k^{(m)}(1, r) \leq c (\omega_mr)^{mk} e^{- (\omega_mr)^m}, 
\]
where $c$ is a positive constant depending only on $m$ and $k$. 
\end{enumerate}
In addition, approximation formulae for $\hat \beta_k^{(m)}(\lambda, r)$, for $m = 2, 3$, were also studied \cite{Robins-2006}.
\begin{figure}
\includegraphics[width=.49\textwidth]{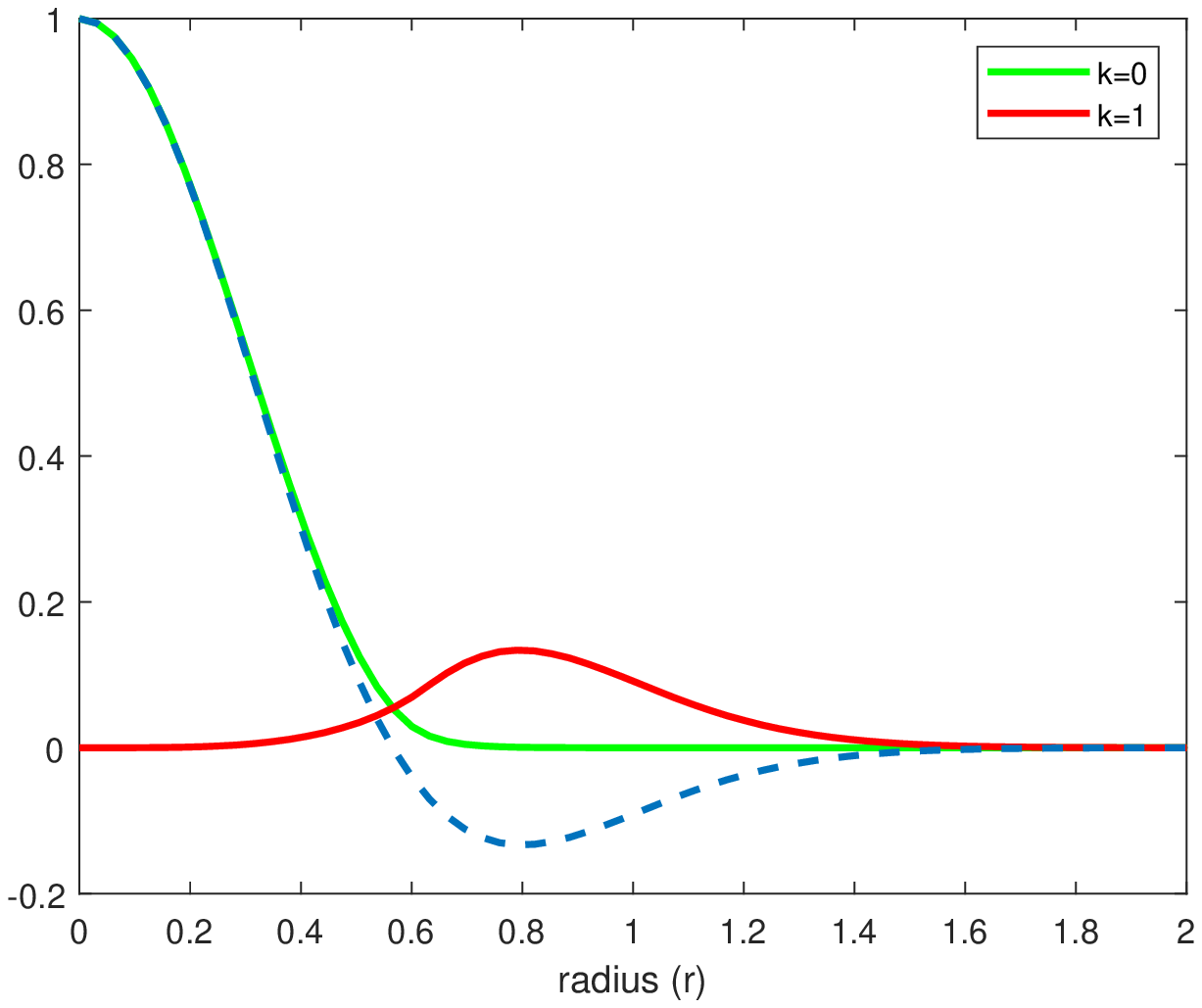}
\includegraphics[width=.49\textwidth]{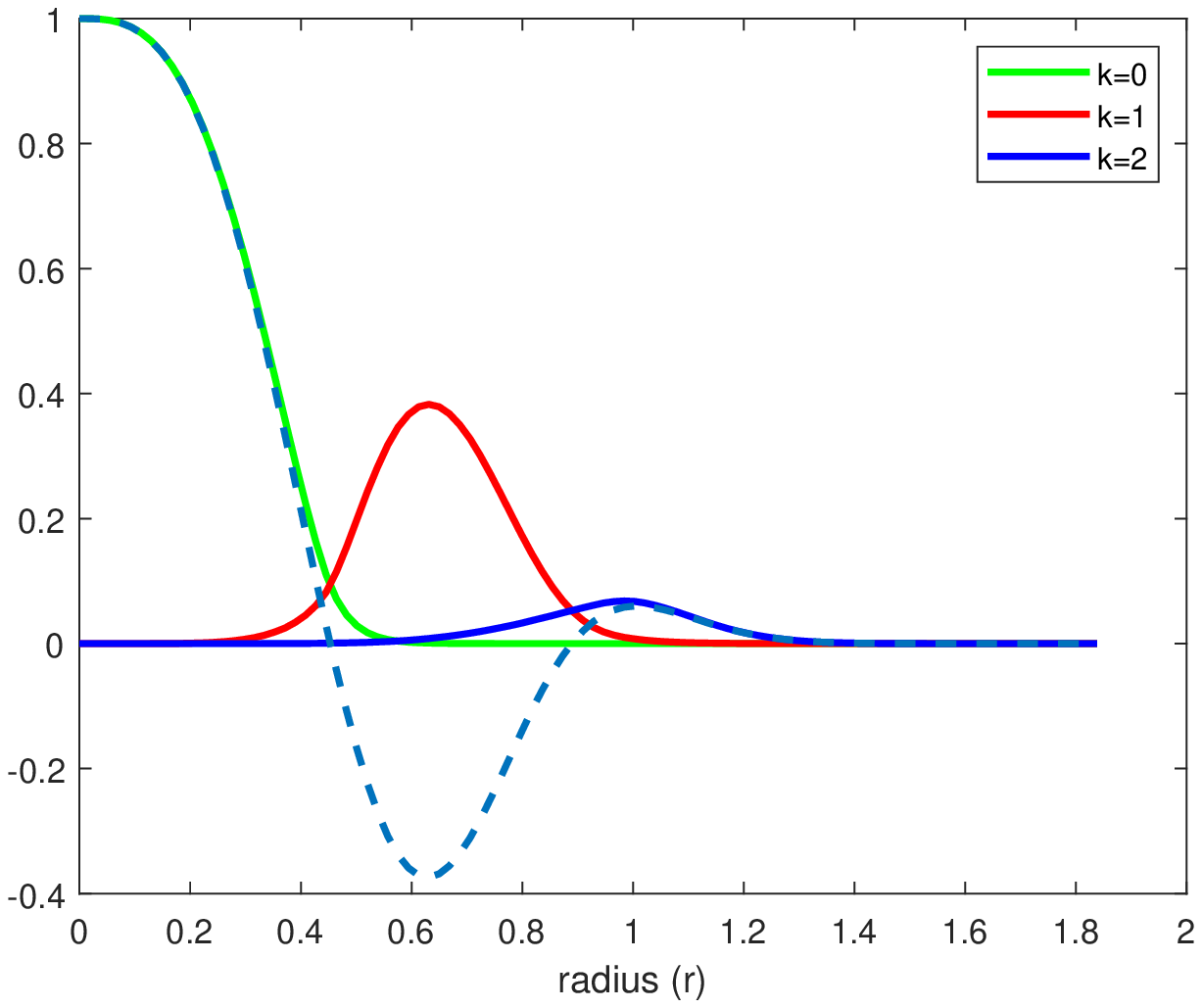}
\caption{Illustration of the limiting Betti numbers $\hat \beta_k^{(N)}(1, r)$ with $N = 2$ (left) and $N = 3$ (right). Plots are numerical values for $n^{-1}\Ex[\beta_k(\cC(n^{1/N}\fX_{n}, r))]$ calculated by taking the average over $20$ times of sampling $n = 10^5$ points uniformly on 
$[0,1]^N$. The dashed lines are for the Euler characteristics, which seem to fit well with the exact value $(1-\pi r^2)e^{-\pi r^2}$ and $\left(\frac{\pi^4 r^6}{6}-4\pi r^3+1\right)e^{\frac{-4\pi r^3}{3}}$ in $2$ and $3$ dimensions, respectively.}
\end{figure}
The scaling, continuity and positivity properties were proved in \cite{Duy-2016,yogi} while the exponential decay property follows easily from the proof of  
\cite[Proposition~6.1]{bobrowski2017random}.

Let us now explain a key idea to deal with the manifold setting. Let $\cM$ be a $C^1$ manifold of dimension $m<N$. Assume for instance that the support of $\kappa$ lies entirely in a single chart $(V, \phi)$ (see Definition~\ref{def_manifolds}), i.e., $\supp(\kappa) \subset \phi(V)$. Then $\{X_i = \phi^{-1}(Z_i)\}_{i \ge 1}$ becomes an i.i.d.~sequence of random variables on $V \subset \R^m$. Moreover, if we define the metric $\rho$ on $V$ by $\rho(x,y) = \|\phi(x) - \phi(y)\|$, then it is clear that $ \cC(\cZ_n, r)$ is identical with $\cC(\fX_n, r, \rho)$, the \v{C}ech complex of radius $r$, constructed on $\fX_n \subset \R^m$ using $\rho$. Thus, the problem on a manifold is converted to that on the Euclidean setting with a general metric $\rho$, which is easier for us to handle. (A general result is stated in Theorem~\ref{thm:euclidb}.)
In general, the support of $\kappa$ may not be covered by a single chart but finitely many ones because of the compactness of $\cM$. As we shall see later, the desired limit theorem in the manifold setting can be derived by taking a suitable partition of the manifold, together with  the spatial independence property of Poisson point processes and the finite additivity of Betti numbers.

From the applications point of view, considering only homology is not enough. 
 It is important to see how persistent the `holes' are, which constitutes the theory of persistent homology.
A good introduction of this topic can be found in  \cite{edel}.
Although we do not discuss about the persistent homology in this article, our results for Betti numbers can be easily extended 
to persistent Betti numbers, and hence to persistence diagrams due to \cite{dhs} (see Remark~\ref{remark_persistent homology}).

The paper is organized as follows. 
Since the main result in the manifold setting is already mentioned here, we state only the main result in the Euclidean setting in the next section. 
In Sections~\ref{section_for_simplex} and~\ref{betti_numbers}, we derive the strong law of large numbers for simplex 
counts and Betti numbers in the Euclidean setting respectively. What we mean by simplex counts and why the strong law for them is needed, will be clear in the 
next section. Finally, we conclude this article by giving the proof of Theorem~\ref{thm:manifoldb} in Section~\ref{simplex-manifold}.

\section{Main Results}\label{main_results}
Let us introduce the definition of Poisson point processes.
Let $\lambda(x) \ge 0$ be a locally integrable function in $\R^N$. 
A point process $\cP$ on $\R^N$ is said to be a Poisson point process with intensity function $\lambda(x)$, denoted by $\cP(\lambda(x))$, if it satisfies the following two conditions
\begin{enumerate}[(i)]
	\item for any bounded Borel set $A$, the random variable $\cP(A)$ counting the number of points in $A$ has Poisson distribution with parameter $\Lambda(A) = \int_A \lambda(x)dx$, i.e., 
	\[
		\Prob(\cP(A) = k) = e^{-\Lambda(A)} \frac{\Lambda(A)^k}{k!}, \quad k = 0,1,\dots;
	\]
 
	\item for disjoint bounded Borel sets $A_1, A_2, \ldots, A_n$, the random variables $\cP(A_1)$, $\cP(A_2), \ldots, \cP(A_n)$ are independent.
\end{enumerate}
In case $\lambda(x) \equiv \lambda$, $\cP$ is called a homogeneous Poisson point process with intensity $\lambda$. Poisson point processes on manifolds can be defined in a similar way.
For more on Poisson point processes, one may see \cite{Poisson}.

In the Euclidean setting, we consider \v{C}ech complexes that are constructed using a general metric $\rho$. (The reason was mentioned as the key idea to deal with the manifold setting in the introduction.) Let $\cA$ be a Borel subset of $\R^N$ equipped with a metric $\rho$. Let $B_{\rho}(x, r) = \{y \in \cA\colon \rho(x,y) \leq r\}$ denote the ball centered at $x$ of radius $r$ with respect to the metric $\rho$. Then for a finite set of points $\fX= \{x_1, x_2, \ldots, x_n\} \subset \cA$ and a radius $r>0$, the \v{C}ech complex constructed using $\rho$, denoted by $\cC(\fX, r, \rho)$, is the collection of all non-empty subsets $\sigma$ of $\fX$ for which $ \cap_{x \in \sigma} B_{\rho}(x, r) \neq \emptyset$. Our main result, here, establishes the limiting behavior of $\beta_k(\cC(\fX_n, r_n, \rho))$,
where the common probability density function $f(x) \in L^p(\R^N)$ for all $1\leq p < \infty$ and  is supported on $\cA$. We need some conditions on the metric $\rho$ and require that $Leb^N(\cA) = 0$. Here and from now on, for any subset $A \subset \R^N$, $\partial A$, $Leb^N(A)$, $A^\circ$, $\bar{A}$ and $|A|$ denote its boundary, $N$-dimensional Lebesgue measure, interior, closure and cardinality, respectively.

The metric $\rho$ is required to be locally approximated by a translation invariant metric induced from a weighted $N$-dimensional Euclidean norm, and to be locally comparable to the Euclidean norm. It is also assumed to be continuous to ensure the measurability of functionals defined on $\cC(\fX, r, \rho)$. 
Such properties can be easily checked when $\rho$ is induced from the manifold setting, as we shall do at the end of this section. More specifically, the two properties of the metric $\rho$ are as follows.
\begin{itemize}
\item[(P1)] 
For fixed $ x \in \cA$, let $d_x(y,z) = \|\B_x(y - z)\|$, where $\B_x$ is a positive definite $N\times N$ matrix and the map $x \mapsto \B_x$ is assumed to be measurable.
 Then we assume that for given $\varepsilon >0 $, there is a $\delta=\delta_{x,\varepsilon} > 0$ such that  for $y, z \in \cA$, whenever $y, z \in B(x, \delta) = \{u : \|u - x\| \le \delta\}$,
\[  
   (1- \varepsilon) d_{x} (y,z) \leq \rho(y,z) \leq (1+ \varepsilon) d_{x} (y,z).
\]
\item[(P2)] There exist constants $\delta, c$ and $C >0$ such that for  $y, z \in \cA$, whenever $\|y-z\| \leq \delta$,
\[ 
       c\|y-z\| \leq \rho(y,z) \leq C\|y-z\|.
\]
\end{itemize}
It is clear that the properties (P1) and (P2) together imply that for $x \in \cA^\circ$,
\begin{equation}\label{prop_metric_d}
      c\|y-z\| \leq d_x(y,z) \leq  C\|y-z\|, \quad \text{for all } y, z \in \R^N.
\end{equation}
Indeed, let $\delta, c$ and $C$ be the constants in (P2). Given $x \in \cA^\circ$ and any $\varepsilon \in (0,1)$, take a constant $\delta_{x, \varepsilon} < \delta$ in (P1) for which $B(x, \delta_{x, \varepsilon}) \subset \cA$ also holds. Then for any $y, z \in B(x, \delta_{x, \varepsilon})$, it follows from (P1) and (P2) that 
\[
     \frac{ c}{1+ \varepsilon}\|y-z\| \leq d_{x} (y,z) \leq \frac{C}{1- \varepsilon}\|y-z\|.
\]
Since both the Euclidean norm and the metric $d_x$ are translation invariant and homogeneous, i.e., $d_x(u + \alpha y, u + \alpha z) = |\alpha| d_x(y, z)$, for $u, y, z \in \R^N$ and $\alpha \in \R$, the above inequality holds for any $y, z \in \R^N$. Since $\varepsilon$ is arbitrary, by letting $\varepsilon \to 0$, we arrive at the desired inequality.

In addition, when $\cA$ is compact, (P1) implies (P2). To see this, for $x \in \cA$, let $\delta_x = \delta_{x, 1/2}> 0$ be the constant in (P1) (with $\varepsilon = 1/2$). Then $\cA$ can be covered by finitely many balls $\{B(x_i, \delta_{x_i}/2)\}_{i = 1}^n$ because of the compactness of $\cA$. For each $x_i$, since $\B_{x_i}$ is a positive definite matrix, it is clear that
\begin{equation}\label{min-max}
	 \lambda_{\min}(\B_{x_i}) \|y-z\|	\le d_{x_i}(y,z) 	\le \lambda_{\max}(\B_{x_i}) \|y-z\|.
\end{equation}
Here, $\lambda_{\min}(\A)$ and $\lambda_{\max}(\A)$ denote the minimum and the maximum eigenvalues of a matrix $\A$ respectively. Let 
\[
\delta =  \min_{i = 1,\dots,n}{\frac{\delta_{x_i}}{2}}, \quad c =\frac12 \min_{i = 1,\dots, n}{\lambda_{\min}(\B_{x_i})}, \quad C =\frac32 \max_{i = 1,\dots, n}{\lambda_{\max}(\B_{x_i})}.
\]
Given $y, z \in \cA$ with $\|y - z\|\le \delta$, it follows that $y, z$ must belong to $B(x_i, \delta_{x_i})$ for some $i \in \{1, \dots, n\}$. Then the property (P1) implies that
\begin{equation}\label{P1}
	\frac{1}{2} d_{x_i}(y,z) \le \rho(y, z) \le \frac{3}{2} d_{x_i}(y, z).
\end{equation}
Therefore, the property (P2) holds for $c, C$ and $\delta$ defined above by combining the two estimates \eqref{min-max} and \eqref{P1}.

For simplicity, when any function is defined on a \v{C}ech complex, the symbol $\cC$ is removed except at few places in the article. For example, we write $\beta_k(\fX_n, r, \rho)$ for $\beta_k(\cC(\fX_n, r, \rho))$.
Our main result in the Euclidean setting is as follows. 
\begin{theorem}[For Euclidean spaces]\label{thm:euclidb}
 Let $(\cA, \rho)$ be a metric space, where $\cA$ is a Borel subset of $\R^N$ with $Leb^N(\partial \cA) = 0$ and the metric $\rho$ satisfies the properties \emph{(P1)} and \emph{(P2)}. Assume that the common probability density function
 $f(x)$ is supported on $\cA$ and for all $j \in \N$, $ \int_{\R^N} f(x)^j dx < +\infty$. Then as $n \to \infty$ with $n^{1/N}r_n \to r \in (0, \infty) $, 
\[
        \frac{\beta_k(\fX_n, r_n, \rho)}{n}  \left(\text{resp.~}\frac{\beta_k(\cP_n, r_n, \rho)}{n}\right) \to \int_{\R^N} \hat {\beta}_k^{(N)} \left( \frac{f(x)}{D(x)}, r \right) D(x) dx ~ ~\text{\rm{a.s.}},
\]
where $D(x) = \det(\B_x)$ with $\B_x$ being the positive definite matrix in the property {\rm(P1)}.
\end{theorem}

Other integer valued random variables defined on  \v{C}ech complexes that will play a vital role in establishing our results, are simplex counts
$S_j(\cdot)$. Here for a simplicial complex $\cK$,  $S_j(\cK)$ is the number of $j$-simplices in $\cK$, or the number of elements of cardinality $(j + 1)$ in $\cK$.
There are existing results in stochastic geometry which may be applicable to simplex counts but not to Betti numbers. 
However, we can estimate Betti numbers by simplex counts due to the following lemma.
\begin{lemma} \label{lem:Betti-estimate}	
 Let $\cK,\tilde \cK$ be any two finite simplicial complexes such that $\tilde \cK \subset \cK$. Then for every $k \ge 0$,
	\begin{equation*}\label{Betti-estimate}
		\left| \beta_k(\cK) - \beta_k (\tilde\cK) \right| \le \sum_{j = k}^{k + 1} \big(S_j(\cK) - S_j(\tilde\cK)\big).
	\end{equation*}
\end{lemma}
Lemma~\ref{lem:Betti-estimate} for $k \geq 1$ is proved in {\cite[Lemma~2.2] {ysa}}. 
For more fundamental proof of Lemma~\ref{lem:Betti-estimate} with $k\geq 0$, see \cite{Duy-2016}.

To establish Theorems~\ref{thm:manifoldb} and \ref{thm:euclidb} for binomial point processes,
we use the standard technique of Poissonization from \cite{Penrose-book}: the coupling of a binomial point process with 
a particular type of a Poisson point process.
Suppose for given $\lambda > 0$, $N_\lambda$ is a
 Poisson random variable with parameter $\lambda$ and independent of 
$\{X_i\}_{i\geq 1}$. Then the point process $\cP_\lambda := \{ X_1, \ldots, X_{N_\lambda}\}$, called a Poissonized version of the binomial point processes,
is a  Poisson point process on $\R^N$ with intensity $\lambda f(x)$ in the Euclidean setting or a Poisson point process on $\cM$ with intensity $\lambda \kappa(z)$ in the manifold setting.
We denote this point process with $\lambda = n$
on  $\R^N$ and on $\cM$ by $\cP_n$ and $\cQ_n$ respectively. The reason for considering these Poissonized versions is because of the spatial independence property of Poisson point processes over disjoint sets. Moreover, the difference of the number of simplices, and hence, of Betti numbers, of the \v{C}ech complexes built over $\fX_n$ (resp.~$\cZ_n$) and $\cP_n$ (resp.~$\cQ_n$) is neglectable as $n \to \infty$.

\begin{lemma}\label{conversion_poisson_binomial}
{\rm(a)} Assume the same assumptions on the manifold $\cM$ and on the density function $\kappa(z)$ as in Theorem~{\rm\ref{thm:manifoldb}}. Then as $n \to \infty$ with $n^{1/m} r_n \to r \in (0, \infty)$,
\[
	      \frac{S_j(\cZ_n, r_n)}{n} - \frac{S_j(\cQ_n, r_n)}{n} \to 0~\text{\rm{a.s.}}
	      \]
	      
\noindent{\rm(b)} Assume the same assumptions on the metric space $(\cA, \rho)$ and on the density function $f(x)$ as in Theorem~{\rm\ref{thm:euclidb}}. Then as $n \to \infty$ with $n^{1/N} r_n \to r \in (0, \infty)$,
\[
	     \frac{S_j(\fX_n, r_n, \rho)}{n} - \frac{S_j(\cP_n, r_n, \rho)}{n} \to 0 \text{ \rm{a.s.}}
\]
\end{lemma}

By this lemma, it is sufficient to prove Theorems~\ref{thm:manifoldb} and \ref{thm:euclidb} only for Poisson point processes.
Note that the functions $f(x)$ and $\kappa(z)$ need not to be probability density functions to prove results for Poisson point processes.

%
%
%

\begin{remark}\label{remark_persistent homology}
Recently, the  strong laws of large numbers for persistent Betti numbers, a generalization of those for Betti numbers, and for persistence diagrams have been established in \cite{dhs} for ergodic stationary point processes on $\R^N$. Our results here can be easily extended in this setting, and may
 be of practical interest. For example, we can easily derive the following results for persistent Betti numbers
\[
\frac{\beta_k ^{s,t} (\{\cC(n^{1/N}\cP_n, r, \rho_n)\}_{r \ge 0} )} {n}  \to \int_{\R^N} \hat {\beta}_k ^{s,t}\left( \frac{f(x)}{D(x)} \right) D(x) dx \text{ a.s.~as } n \to \infty, 
\]
and for persistence diagrams 
\[
          \frac{1}{n} \xi_k (n^{1/N}\cP_n, \rho_n) \vto \int_{\R^N} \nu_k\left( \frac{f(x)}{D(x)}\right) D(x) dx  \text{ a.s.~as } n \to \infty,
\]
where $n^{1/N}\cP_n$ is a Poisson point process on $\R^N$ with intensity function $f(x/n^{1/N})$ and $\rho_n$ is a metric on $n^{1/N}\cA$ defined as $\rho_n (x, y) = n^{1/N} \rho(x/n^{1/N}, y/n^{1/N})$. The rest of the notations are taken from \cite{dhs}. 
\end{remark}

\subsection{Motivation for the assumptions on the metric $\rho$}\label{motivation}
Before giving the motivation, we give the definition of a $C^1$ manifold and its boundary, taken from \cite{munkfold}.
Let $\bH^m$ denote the upper half space in $\R^m$, consisting of $x \in \R^m$ for which the $m$th coordinate 
$x_m \geq 0$.

\begin{definition}\label{def_manifolds}
A nonempty subset $\cM$ of $\R^N$, endowed with the subset topology, is called an $m$-\textit{dimensional $C^1$ manifold} 
if for each $z \in \cM$ there exist an open set $M$ of $\cM$ containing $z$, a set $V$ that is open in either $\R^m$ or $\bH^m$, 
and a $C^1$ 
bijective map $\phi \colon V \to M$ such that  
\begin{enumerate}[(i)]
\item{$\phi^{-1}\colon M \to V$ is continuous;}
\item{the Jacobian of $\phi$ at $x$, denoted by $\J_{\phi}(x)$, has rank $m$ for all $x \in V$.}
\end{enumerate}
\end{definition}
The pair $(V, \phi)$ is called a chart for $z$. 
\begin{definition}
Let $(V, \phi)$ be a chart for $z \in \cM$. We say $z$ is a \textit{boundary point} of $\cM$ if $V$ is open in $\bH^{m}$ and $z  = \phi(x)$ for $x \in \R^{m-1}\times \{0\}$.
\end{definition}

Let $(V, \phi)$ be a chart and $U$ be a convex and compact subset of $\R^m$ such that $U \subset V$. Let $U$ be equipped with the metric 
$\rho(y, z):= \|\phi(y) - \phi(z)\|$. We will show that the metric $\rho$ on $U$ satisfies the two properties (P1) and (P2). Note that we consider now a subset of $\R^m$ (not $\R^N$).  Since $U$ is compact, it is sufficient to verify (P1).

Let $\phi = (\phi_{1}, \phi_{2}, \ldots, \phi_{N})$, where for all $j \in \{1,2, \ldots, N\}$,  $\phi_j$ is a $C^1$ function from $V$ to $\R$. 
Let $\nabla \phi_j (x)$ denote the derivative of $\phi_j$. Since $\phi_j$ is a $C^1$ function, 
for given $x \in U$ and $\varepsilon >0 $, there exists $\delta_j >0$ such that
whenever $ y \in B(x, \delta_j) \cap U$,
\begin{equation*}
      \| \nabla \phi_j (y) - \nabla \phi_j(x) \| \leq \varepsilon.
\end{equation*} 
Let $\delta = \min_{j} \delta_j$ and $ y, z \in B(x, \delta)\cap U$.
 Let $(y, z)$ denote the open line segment joining $y$ and $z$. By the mean value theorem for each $\phi_j$, 
there exists $c_j \in (y, z) \subset U$ such that 
\[
    \phi_j (y) - \phi_j(z) = \nabla \phi_j (c_j) (y - z).
\]
Therefore,
\begin{equation*} \label{nabla_2}
 \rho(y, z) = \|\C (y-z)\|,
\end{equation*}
where $\C$ is the $N \times m$ matrix whose $j$th row is equal to $\nabla \phi_j(c_j)$. 

Let $\B_{x} = \J_{\phi} (x)$ and $d_x(y, z) = \|\B_{x}(y - z)\| = \|(\B_{x}^t \B_{x})^{1/2} (y - z)\|$.
Then by the triangle inequality,
\begin{equation}\label{norm_ineq}
|\rho(y, z) - d_x(y,z)| =
  \big| \|\C (y-z) \| - \| \B_{x} (y-z)\| \big| \leq \| (\C - \B_{x}) (y - z)\|.
\end{equation}
Note that for an $N \times m$ matrix $\B=({b_{i,j}})$, it holds that 
\[
	\|\B v\|^2 \le \Big (\sum_{i, j} b_{i, j}^2 \Big)  \|v\|^2.
\]
Therefore, 
\[
	\| (\C - \B_{x}) (y-z)\|^2 \le \Big( \sum_{j=1}^{N} \|\nabla \phi_j(c_j) - \nabla \phi_j (x_0)\|^2 \Big) \|y - z\|^2 \le N \varepsilon^2 \|y - z\|^2.
\]
This inequality together with the inequality \eqref{norm_ineq} implies that 
\[
	|\rho(y, z) - d_x(y,z)| \le \sqrt{N} \varepsilon \|y - z\|.
\]
The verification is complete by taking into account of the inequality \eqref{min-max}.

\section{Simplex counts in the Euclidean setting}\label{section_for_simplex}
In this section, we show the strong law of large numbers for simplex counts $S_j(\cdot)$ in the Euclidean setting
and give the proof of the statement (b) of Lemma~\ref{conversion_poisson_binomial}. 
Note that the strong law for $S_j(\cdot)$ in the thermodynamic regime may follow from the general theory of local functionals due to Penrose and Yukich \cite{pen2007law, pen2003weak}. However, we present here an elementary proof by  calculating the order of the fourth moments.

Our argument is based on the following two results, called Palm theorems, to calculate the expectation and higher moments of simplex counts. Let $\cP(f(x))$ be a Poisson point process on $\R^N$ with intensity function $f(x)$, which is assumed to be integrable over $\R^N$. 
\begin{theorem}[{cf.~\cite[Theorem~1.6]{Penrose-book}}] \label{palm1} Suppose $j \in \N$ and $h(\cY)$ is a bounded measurable function defined on all 
finite subsets $\cY$ of $\R^N$, satisfying $h(\cY) = 0$ except when $\cY$ has $j$ elements. Then 
\[
	\Ex \bigg[\sum_{\cY \subset \cP(f(x))} h(\cY) \bigg] = \frac{1}{j!} \int_{(\R^N)^j} h(\y) \prod_{i=1}^j f(y_i) d\y,
\]
where, for $\y= (y_1, y_2, \ldots, y_j)\in (\R^N)^j$, we use the shorthands $h(\y) := h(y_1, y_2, \ldots, y_j)$ and $d\y := dy_1\cdots dy_j$.
\end{theorem}
\begin{theorem}[{cf.~\cite[Theorem~1.7]{Penrose-book}}] \label{palm2}
Suppose $k \in \N$. Then under the same assumptions on the function $h(\cY)$ as in \emph{Theorem \ref{palm1}},
\begin{align*}
		&\Ex\bigg[ \sum_{\cY_{1} \subset \cP(f(x))} \cdots \sum_{\cY_{k} \subset \cP(f(x))}\Big( \prod_{i = 1}^k h(\cY_{i}) \Big) \one_{\{\cY_{i} \cap \cY_{j} = \emptyset \text{ for } 1 \le i < j \le k \}}  \bigg] \\
		&= \prod_{i = 1}^k \Ex\bigg[ \sum_{\cY_i \subset \cP(f(x))} h(\cY_{i})\bigg],
	\end{align*}
where $\one_{\{\cY_{i} \cap \cY_{j} = \emptyset \text{ for } 1 \le i < j \le k \}}$ is the indicator function that is equal to $1$ if and only if $\cY_{i} \cap \cY_{j} = \emptyset \text{ for } 1 \le i < j \le k$.
\end{theorem}

Let $(\cA, \rho)$ satisfy the assumptions in Theorem~\ref{thm:euclidb}. Let $h_{j, r_{n}, d} (\cY)$ be the indicator function which is equal to $1$ if and only if $\cY \subset \R^N$ is a $j$-simplex of radius $r_n$, measured using the metric $d$. 
 When $r_n=1$ and $d$ is the
Euclidean metric, we write $h_j (\cY) $ for $h_{j, 1, \left\|\cdot\right\|} (\cY)$. Then the number of $j$-simplices in the \v{C}ech complex $\cC(\fX, r, \rho)$ can be written as 
\begin{equation}\label{sim_form}
  S_j (\fX, r_n, \rho) = \sum_{\cY \subset \fX} h_{j, r_n, \rho} (\cY),
\end{equation}
where $\fX \subset \cA$ is a finite set.

For $r \in (0, \infty)$ and $\x= (x_1, x_2, \ldots, x_j)\in (\R^N)^j$, define
\[
      A_j^{(N)}(r) = \frac{r^{Nj}}{(j+1)!} \int_{(\R^N)^j}  h_j (0, \x) d\x,
\] 
where $h_j (0, \x)$ and $d\x$ stand for $h_j (0, x_1, x_2, \ldots, x_j )$ and $dx_1\cdots dx_j$ respectively.
   
Recall that $\cP_{n}$ is the Poisson point process on $\R^N$ with intensity function $nf(x)$, where $f(x)$ is supported on $\cA$ and $\int_{\R^N} f(x) dx < \infty$. 
We now show the right order of the expectations  $\Ex[S_j (\cP_{n}, r_n, \rho)]$ when $r_n$ tends to zero. Then by calculating the order of the fourth central moments, we shall prove that 
in the thermodynamic regime, $n^{-1}(S_j(\cP_n, r_n, \rho) - \Ex[S_j(\cP_n, r_n, \rho)])$ converges to zero almost surely as $n$ tends to infinity. More explanation on the fourth moments will be given later.

\begin{proposition}\label{general_simplicial}
Assume that  $\int_{\cA} f(x)^{j+1} dx < +\infty$ and $\lim _{n \to \infty} r_n =0 $. Then
\[
   \lim_{n \to \infty} r_n^{-Nj} n ^{-(j+1)} \Ex[S_j (\cP_{n}, r_n, \rho)] = A_j^{(N)}(1) \int_{\cA} \frac{f(x)^{j+1}}{D(x)^j} dx.
\]
\end{proposition}

It follows from the equation \eqref{sim_form} and Theorem~\ref{palm1} that
\begin{align*}
  & r_n^{-Nj} n ^{-(j+1)} \Ex[S_j (\cP_{n}, r_n, \rho)]  \\
  &= \frac{r_n^{-Nj}}{(j+1)!} \int_{\cA^{j+1}} h_{j, r_{n}, \rho} (x_0, \x) \prod_{i=1}^j f(x_i) f(x_0) d\x d{x_0}\\
   &=  \frac{r_n^{-Nj}}{(j+1)!} \int_{\cA}\left(\int_{\cA^{j}} h_{j, r_{n}, \rho} (x_0, \x)d\x\right) f(x_0)^{j+1} dx_0 \\
  &\quad +  \frac{r_n^{-Nj}}{(j+1)!} \int_{\cA^{j+1}} h_{j, r_{n}, \rho} (x_0, \x) \left( \prod_{i=1}^j f(x_i) -  f(x_0)^{j} \right)f(x_0) d\x dx_0\\
  &= I_1^n+ I_2^n.
\end{align*}
We claim that $I_1^n$ is asymptotic to $ A_j^{(N)}(1) \int_{\cA} \frac{f(x)^{j+1}}{D(x)^j} dx$, while $I_2^n$ tends to zero as $n$ tends to infinity, 
from which Proposition~\ref{general_simplicial} follows. Our claims are proved in the next two lemmas.

\begin{lemma}\label{lem: ex_conv_sim}
As $n \to \infty$,
 \[
      I_1^n =  \frac{r_n^{-Nj}}{(j+1)!}  \int_ {\cA} \left( \int _{\cA^j}  h_{j, r_{n}, \rho} (x_0, \x) d\x \right)f(x_0)^{j+1} dx_0 \to A_j^{(N)}(1) \int_{\cA} \frac{f(x)^{j+1}}{D(x)^j} dx.
 \] 
\end{lemma} 
\begin{proof}
The idea is to use the Dominated Convergence Theorem (DCT) for the integral with respect to $dx_0$. We first show that for $n$ large enough the integrand is dominated by an integrable function and then compute the point-wise limit.
Let 
\[ 
       F_n (x_0) = r_n^{-Nj} \int _{\cA^j}  h_{j, r_{n}, \rho} (x_0, \x) d\x.
\]
Since the metric $\rho$ satisfies the property (P$2$) and $r_n \to 0$ as $n \to \infty$, there exists $n_0 \in \N$ such that for all $n \geq n_0$, 
$h_{j, r_{n}, \rho} (x_0, \x) \leq  h_{j, \tilde{r}_{n}} (x_0, \x)$, where $\tilde{r}_n = r_n/c$ with $c$ being the constant in (P$2$). Therefore, 
\[
    F_n (x_0) \leq r_n^{-Nj} \int _{\cA^j}  h_{j, \tilde{r}_{n}} (x_0, \x) d\x.
\]
Applying the change of variables $x_i = x_0+\tilde{r}_n y_i$ for $1\leq i \leq j$ yields
\[
     F_n (x_0) \leq  c^{-Nj}\int _{B(0,2)^j}  h_{j} (0, \y) d\y \leq (c^{-N}Leb^N(B(0,2)))^j.
\]
Thus for $n$ large enough, the integrand is dominated by $const\cdot f(x_0)^{j+1}$ which is integrable by our assumption.

Next we consider the point-wise limit of $F_n$. Let  $x_0 \in \cA^\circ$. Since $\rho$ satisfies (P$1$) and $r_n \to 0$, 
given $\varepsilon > 0$, there exists $n_{x_0} \in \N$ such that for all $n \geq n_{x_0}$,
\begin{equation}\label{relationh_j}
       h_{j, r_{n}^+, d_{x_0}} (x_0, \x)  \leq  h_{j, r_{n}, \rho} (x_0, \x) \leq h_{j, r_{n}^-, d_{x_0}} (x_0, \x),
\end{equation}
where $r_n^+ = r_n/(1+ \varepsilon)$ and $r_n^- = r_n/(1- \varepsilon)$.
Let 
\[
   I_n^{+\varepsilon}(x_0) =   r_n^{-Nj} \int _{\cA^j}  h_{j, r_{n}^+, d_{x_0}} (x_0, \x) d\x ,\quad  I_n^{-\varepsilon}(x_0) = r_n^{-Nj} \int _{\cA^j}  h_{j, r_{n}^-, d_{x_0}} (x_0, \x) d\x.
\]
Then from the relation \eqref{relationh_j}, we have that
\[ 
     \lim_{\varepsilon \to 0} \lim_{n \to \infty} I_n^{+\varepsilon}(x_0) \leq \liminf_{n\to \infty} F_n (x_0) \leq \limsup_{n\to \infty} F_n (x_0) \leq  \lim_{\varepsilon \to 0} \lim_{n \to \infty} I_n^{-\varepsilon}(x_0).
\]
Consider $I_n^{+\varepsilon}(x_0)$.
By applying the change of variables $x_i = x_0+r_n^+ y_i$ for $1\leq i \leq j$, the indicator function 
$h_{j, r_{n}^+, d_{x_0}} (x_0, x_0+r_n^+ y_1, \ldots, x_0+r_n^+ y_j)$ is equal to $h_{j, 1, d_{x_0}} (0, \y)$ for all large enough $n$ since $x_0 \in \cA^\circ$ and $r_n \to 0$.
Thus, for $x_0 \in \cA^\circ$,
\[
\lim_{n \to \infty} I_n^{+\varepsilon}(x_0) = (1+ \varepsilon)^{-Nj} \int _{(\R^N)^j}  h_{j, 1, d_{x_0}} (0, \y) d\y =  A_j^{(N)}(1)  \frac{(1+ \varepsilon)^{-Nj}}{ D(x_0)^j}.
\]
Similarly, for $x_0 \in \cA^\circ$,
\[
   \lim_{n \to \infty} I_n^{-\varepsilon}(x_0) =  A_j^{(N)}(1) \frac{(1- \varepsilon)^{-Nj}}{ D(x_0)^j}.
\]
Therefore, $F_n(x_0)$ converges to
$A_j^{(N)}(1)/D(x_0)^j$ almost everywhere because of the assumption $Leb^N(\partial \cA)=0$, from which the desired result follows. 
\end{proof}

\begin{remark}
\begin{enumerate}[(i)]
\item{If $\rho$ is homogeneous and translation invariant, the change of variables can be applied directly in $F_n (x_0)$.} 
\item{If $\cA = \R^N$ then for $x_0 \in \R^N$, the function 
$h_{j, {r}_{n}^+, d_{x_0}} (x_0, x_0+ {r}_n^+ y_1, \ldots, x_0+{r}_n^+ y_j)$ is equal to $h_{j, 1, d_{x_0}} (0, \y)$ for all $n \in \N$.}
\end{enumerate}
\end{remark}

 \begin{lemma}\label{lem: ex_conv_sim2}
As $n\to \infty$,
\[
    I_2^n = \frac{r_n^{-Nj}}{(j+1)!} \int_{\cA^{j+1}} h_{j, r_{n}, \rho} (x_0, \x) \left( \prod_{i=1}^j f(x_i) -  f(x_0)^{j} \right)f(x_0) d\x dx_0 \to 0.
  \]
 \end{lemma}
 \begin{proof}
 The absolute value of $I_2^n$ is bounded by 
 \[
        \int_{\R^N} \left(\int_{B_\rho(x_0, 2r_n)^{j}} r_n ^{-Nj} \left| \prod_{i=1}^j f(x_i) -  f(x_0)^{j} \right| d\x\right) f(x_0) dx_0.
 \]
 We claim that the above expression tends to zero as $n$ tends to infinity. To show this, DCT is used again. Let 
 \[
      F_n(x_0) = \left(\int_{B_\rho(x_0, 2r_n)^{j}} r_n ^{-Nj} \left| \prod_{i=1}^j f(x_i) -  f(x_0)^{j} \right| d\x\right) f(x_0).
\]
Let us first show that $F_n$ is uniformly bounded in $n$ by an integrable function. Clearly, 
\[
      F_n(x_0) \leq r_n^{-Nj} f(x_0) \int_{B_\rho(x_0, 2r_n)^{j}} \prod_{i=1}^j f(x_i)  d\x +  r_n^{-Nj} f(x_0)^{j+1} \left(Leb^N(B_{\rho}(x_0, 2r_n))\right)^j.
\] 
Since $\rho$ satisfies (P$2$), for all large enough $n$, $B_{\rho}(x_0, 2r_n) \subset B(x_0, 2r_n/c)$.
Therefore, with $C =\left( Leb^N(B(0, 2/c))\right)^j$, we have
\[
      F_n(x_0) \leq C f(x_0) \left( \frac{1}{Leb^N(B(x_0, 2r_n/c))} \int_{B(x_0, 2r_n/c)} f(x)  dx\right)^j + Cf(x_0)^{j+1}.
\]
 Let $(Mf)(x)$ be the centered Hardy--Littlewood maximal function, i.e., 
\[
      (Mf)(x_0) = \sup \frac{1}{Leb^N(B(x_0, r))} \int_{B(x_0, r) } f(x)  dx, 
\]
where the supremum is taken over all the balls in $\R^N$ whose center is $x_0$. Then
\[
      F_n(x_0) \leq C f(x_0) (Mf)(x_0)^j + Cf(x_0)^{j+1}.
\]
Thus, we only need to show that the function $f(x_0) (Mf)(x_0)^j$ is integrable to conclude that $F_n$ is dominated by an integrable function for $n$ large enough. Since $f \in L^p(\R^N)$ for $1<p\leq \infty$ implies $Mf \in L^p(\R^N)$, 
by H\"{o}lder's inequality with $p = j+1$ and $q = (j+1)/j$, 
\[
     \int_{\R^N} f(x_0) (Mf)(x_0)^j d{x_0} \leq \left( \int_{\R^N} f(x_0)^{j+1} d{x_0} \right)^{\frac{1}{j+1}}\left(  \int_{\R^N} (Mf)(x_0)^{j+1} d{x_0}\right)^{\frac{j}{j+1}} < \infty.
\]

It remains to show that $F_n(x_0)$ converges to zero almost everywhere. 
Let $x_0$ be a Lebesgue point of $f$ such that $f(x_0) < \infty$.  The convergence is proved by the induction on $j$.
The inductive step is to bound $F_n (x_0)$ by  
\begin{align*}
      \bigg(\int_{B(x_0, 2r_n/c)^{j}} r_n ^{-Nj} &|f(x_j) - f(x_0)| \prod_{i=1}^{j-1} f(x_i) d\x \bigg) f(x_0)\\
       &+   \left(\int_{B(x_0, 2r_n/c)^{j}} r_n ^{-Nj} \left| \prod_{i=1}^{j-1} f(x_i) - f(x_0)^{j-1} \right| f(x_0) d\x \right) f(x_0).
\end{align*}
The first term with $V_n = Leb^N(B(x_0, 2r_n/c))$ can be written as
\[
     C \left(\frac{1}{V_n} \int_{B(x_0, 2r_n/c)} f(x) dx \right)^{j-1} \left( \frac{1}{V_n} \int_{B(x_0, 2r_n/c)} |f(x_j) - f(x_0)| dx_j\right) f(x_0),
\]
which converges to zero since $x_0$ is a Lebesgue point of $f$ with $f(x_0)< \infty$. The second term also tends to zero by the inductive hypothesis.
Hence, by the Lebesgue differentiation theorem and the almost everywhere finiteness of $f$, the function $F_n$ tends to zero almost everywhere. 
This completes the proof.
 \end{proof}

In the thermodynamic regime, Proposition~\ref{general_simplicial} is restated as follows.
\begin{corollary}\label{corollary}
Assume that $\int_{\cA} f(x)^{j+1} dx < +\infty$ and $ \lim _{n \to \infty} n^{1/N} r_n = r \in (0, \infty)$. Then
\[
   \lim_{n \to \infty}  \frac{\Ex[S_j (\cP_n, r_n, \rho)]} {n} = A_j^{(N)}(r) \int_{\cA} \frac{f(x)^{j+1}}{D(x)^j} dx.
\]
\end{corollary}

By defining $\hat S_j^{(N)} \left(\lambda, r\right) := A_j^{(N)}(r) \lambda^{j+1}$, the limiting constant in the above corollary can also be written as $\int_{\cA} \hat S_j^{(N)} \left(f(x)/D(x), r\right) D(x) dx$. 
Note that $\hat S_j^{(N)} \left(\lambda, r\right)$ is nothing but the limiting constant of $L^{-1}\Ex[S_j(\cP_L(\lambda), r)]$ when $L \to \infty$ (for the definition of $\cP_L(\lambda)$, see the introduction). 
The following corollary is also a result on homogeneous Poisson point processes but with the metric $d_{x_0}$, which is needed in Section~\ref{compact_region}. 
\begin{corollary}\label{corhom}
Let $x_0 \in \cA^\circ$. Let $E \subset \R^N$ be a Borel set with $Leb^N(E) > 0$, and let $E_n = n^{1/N}E$. Then for any $\lambda \geq 0$ and $r \in (0, \infty)$, 
\[
 \lim_{n \to \infty}  \frac{\Ex[S_j (\cP(\lambda)|_{E_n}, r, d_{x_0})]} {Leb^N(E_n)} = \hat S_j^{(N)} \left(\frac{\lambda}{D(x_0)}, r\right) D(x_0),
\]
where $\cP(\lambda)|_{E_n}$ is the restriction of $\cP(\lambda)$ on $E_n$.
\end{corollary}
\begin{proof}
Take $f(x) = \lambda$, $\cA = E$ and $r_n = r/n^{1/N}$ in Corollary \ref{corollary}, we get 
\begin{equation} \label{eq1hom}
       \lim_{n \to \infty} \frac{ \Ex[S_j (\cP_{n}, r_n, d_{x_0})] } {n}= A_j^{(N)}(r) Leb^N(E) \frac{\lambda^{j+1}}{ D(x_0)^j}.
\end{equation}
Consider the map $x \to n^{1/N}x$ with $n^{1/N}r_n = r \in (0, \infty)$ on the set $E$. By the scaling property of Poisson point processes, we have  
\begin{equation}\label{eq2hom}
            \Ex[S_j (\cP_{n}, r_n, d_{x_0})] = \Ex[S_j (\cP(\lambda)|_{E_n}, r, d_{x_0})].
\end{equation}
Substituting the equation \eqref{eq2hom} into \eqref{eq1hom} yields
\[
 \lim_{n \to \infty}  \frac{\Ex[S_j (\cP(\lambda)|_{E_n}, r, d_{x_0})]} {Leb^N(E_n)} =  A_j^{(N)}(r)\frac{\lambda^{j+1}}{D(x_0)^j},
\]
which completes the proof.
\end{proof}

Let $\cY_1, \cY_2, \ldots, \cY_k \subset \cP_n$ with $|\cY_i| =j+1, (1\leq i\leq k)$. 
Note that  Theorem~\ref{palm2} can be applied only when the sets $\{\cY_i\}_{i = 1}^k$ are disjoint. In order to calculate the fourth moment of $S_j(\cP_n, r_n, \rho)$, we need to deal with the non-disjoint case.
Let us consider the case when $k=2$, and leave to the reader the straightforward generalization to higher values of $k$.

\begin{lemma}\label{second order}
Assume that $\int_{\cA} f(x)^{2j+1} dx< +\infty$ and  $\lim _{n \to \infty} n^{1/N} r_n = r \in (0, \infty)$.  Then for any $1 \le l \le j + 1$, 
\[
   \lim_{n \to \infty} \frac{1}{n} \Ex \bigg[\sum_{\cY_1 \subset \cP_n} \sum _{\cY_2 \subset \cP_n}   h_{j,r_n,\rho}{(\cY_1)} h_{j,r_n,\rho}{(\cY_2)} \one_{\{|\cY_1 \cap \cY_2| = l\}}\bigg]  = const.
\]
\end{lemma} 
\begin{proof}
For each finite subset $\cZ \subset \R^N$, let 
\[
   h'_{j,r_n,\rho}(\cZ) = \one_{\{ |\cZ| = 2j+2-l\}}  \sum_{\cY \subset \cZ} \sum _{\cY' \subset \cZ} h_{j,r_n,\rho} (\cY)h_{j,r_n,\rho}(\cY') \one_{\{\cY \cup \cY' = \cZ\}}. 
\]
Then 
\[
\Ex \bigg[\sum_{\cY_1 \subset \cP_n} \sum _{\cY_2 \subset \cP_n}   h_{j,r_n,\rho}{(\cY_1)} h_{j,r_n,\rho}{(\cY_2)} \one_{\{|\cY_1 \cap \cY_2| = l\}}\bigg] = 
\Ex \bigg[\sum_{\cZ\subset \cP_n} h'_{j,r_n,\rho} (\cZ)\bigg],
\]
from which all the arguments in the proof of Proposition~\ref{general_simplicial} can work here to show the desired result.
\end{proof}

We are now in the position to state and prove the strong law of large numbers for simplex counts.

\begin{proposition}\label{thm: 4-order}
Assume that  $\int_{\cA} f(x)^{4j+1} dx< +\infty$ and $\lim_{n \to \infty}n^{1/N}r_n = r \in (0, \infty)$. Then as $n \to \infty$, 
\[
 \frac{S_j(\cP_n, r_n, \rho)} {n} \to A_j^{(N)}(r) \int_{\cA} \frac{f(x)^{j+1}}{D(x)^j} dx ~ \text{\rm{a.s.}}
\]
\end{proposition}

This section is concluded with the proofs of Proposition~\ref{thm: 4-order} and the statement (b) of Lemma~\ref{conversion_poisson_binomial}.
In both the proofs, we show that the fourth moments of the appropriate quantities are of $O(n^\delta)$, where $0 \leq \delta < 3$ and $O(\cdot)$ is Bachmann--Landau big-$O$ notation. This is a sufficient condition for the strong laws. 
 For instance, let $\xi_n = (S_j(\cP_n, r_n, \rho) - \Ex[S_j(\cP_n, r_n, \rho)])$. We shall show that   
 $\Ex[\xi_n^4] \leq Kn^2$, for some constant $K$. Then by Markov's inequality, $\Prob(|\xi_n| \geq n\varepsilon)  \leq K n^{-2}\varepsilon^{-4}$. Since 
 $\sum n^{-2} < \infty$, by the first Borel--Cantelli lemma, $\Prob(\limsup_n |n^{-1} \xi_n| \geq \varepsilon) =0$. 
 This means $n^{-1} \xi_n$ converges to zero 
 almost surely. The strong law for $n^{-1}S_j(\cP_n, r_n, \rho)$ in the thermodynamic regime then follows from the convergence of expectation, which has been established in Corollary~\ref{corollary}. Now let us get into more detail on the proof of Proposition~\ref{thm: 4-order}.

\begin{proof}[{Proof of Proposition~{\rm\ref{thm: 4-order}}}]
The presentation of this proof is similar to some parts of the proof of Theorem~$3.9$ in \cite{Penrose-book}. 
To ease notation, we write $S_j$ for $S_j(\cP_n, r_n, \rho)$ and $h'_j$ for $h_{j,r_n,\rho}$. Using the binomial expansion, we have
\begin{align}
\label{eqn: binomial}
   \Ex [(S_j - \Ex[S_j])^4] = \sum_{k=0}^{4} {4 \choose k} (- \Ex[S_j])^{4-k}  \Ex[S_j ^ k], 
\end{align}
where 
\begin{equation}
\label{eqn:sum_k}
  \Ex[S_j ^k] = \Ex \left[\sum_{\cY_1 \subset \cP_n} \sum _{\cY_2 \subset \cP_n} \cdots \sum_{\cY_k \subset \cP_n}  h'_j{(\cY_1)} h'_j{(\cY_2)} \cdots h'_j{(\cY_k)}\right]. 
\end{equation}
Note that $\Ex[S_j ^k] = O(n^k)$. So at the first look it seems that the equation
  \eqref{eqn: binomial} is of $O(n^4)$. 
 However, this order is not exact. To calculate the exact order, we express $ \Ex [S_j^k]$ in \eqref{eqn:sum_k} as a sum of finitely many terms according to the intersection of $\{\cY_i\}_{i = 1}^k$ and analyze the order of each term. For instance, when $k = 2$,  there are $(j + 2)$ terms corresponding to the conditions $|\cY_1 \cap \cY_2| = l, (l=0, \dots, j+1)$. The term corresponding to $l = 0$ coincides with $\Ex[S_j]^2$ by Theorem~\ref{palm2}, and hence, has order $n^2$, while the other terms have order $n$ by Lemma~\ref{second order}. In general, each term in the expression~\eqref{eqn:sum_k} has order $n^{k'}$, where $k' \le k$ is the number of disjoint components in $\{\cY_i\}_{i = 1}^k$.

 The only term in \eqref{eqn:sum_k} that gives  an order $n^k$ comes from the case where all $\cY_1, \cY_2, \ldots, \cY_k$ are disjoint. By Theorem~\ref{palm2}, this term is equal to $\Ex[S_j]^k$. Putting back this contribution for each $k$ in \eqref{eqn: binomial}, we see that the coefficient of the leading order  term, the term of order $n^4$, is zero.

Now consider terms of order $n^{k-1}$ in \eqref{eqn:sum_k}. They should come from ordered 
$k$-subsets $\cY_1, \cY_2, \ldots, \cY_k$ when two of them have $1 \leq l \leq j+1$ points in common 
and the remaining subsets have neither any point in common with each other nor 
with the two subsets. Clearly, in this case there is no contribution when $k=0$ or $k=1$.
For fixed $l$ and $2\leq k\leq 4$, this contribution, denoted by $T^{k,l}$, is 
\begin{align*}
T^{k,l} &=  {k \choose 2}\Ex \bigg[\sum_{\cY_1 \subset \cP_n} \sum_{\cY_2 \subset \cP_n} h'_j{(\cY_1)} h'_j{(\cY_2)} \one_{\{|\cY_1\cap \cY_2| = l\}}\bigg]
\prod_{s=3}^{k}\Ex \bigg[\sum_{\cY_s \subset \cP_n}h'_j{(\cY_s)}\bigg]\\
&= {k \choose 2}\Ex \bigg[\sum_{\cY_1 \subset \cP_n} \sum_{\cY_2 \subset \cP_n} h'_j{(\cY_1)} h'_j{(\cY_2)} \one_{\{|\cY_1\cap \cY_2| = l\}}\bigg] \Ex[S_j]^{k-2}.
\end{align*}
Putting back the contribution $T^{k,l}$ for $k \geq 2$ in \eqref{eqn: binomial} again makes all terms of order $n^3$ disappear. Thus, we deduce that $\Ex[(S_j - \Ex[S_j])^4] \le K n^2$, for some constant $K$. The proof is complete.
\end{proof}

\begin{remark}\label{euclid_simp_hom}
  Similarly, under the setting of Corollary \ref{corhom}, as $n \to \infty$,
 \begin{equation*}
  \frac{S_j (\cP(\gamma)|_{E_n}, r, d_{x_0})} {Leb^N(E_n)} \to \hat S_j^{(N)} \left(\frac{\gamma}{D(x_0)}, r\right) D(x_0) ~\text{a.s.}
 \end{equation*} 
\end{remark}

We now present the proof of the statement (b) of Lemma~\ref{conversion_poisson_binomial}. The statement (a) is for the manifolds, so its proof is discussed in 
Section \ref{simplex-manifold}.
\begin{proof}[{Proof of Lemma \rm\ref{conversion_poisson_binomial} \rm{(b)}}]
For any $m$, let $S_j(m, n) = |S_j(\fX_m, r_n, \rho) - S_j (\fX_n, r_n, \rho) |$.
We first bound the probability of the event $\{X_1 \in B_{\rho}(x, r_n) \cap \cA\}$. Clearly,
\begin{equation}\label{bound_prob}
      \Prob(X_1 \in B_{\rho}(x, r_n)\cap \cA)  = \int_{B_{\rho}(x, r_n)\cap \cA} f(x) dx.
 \end{equation}
Applying H\"{o}lder's inequality, we obtain that 
\[
      \Prob(X_1 \in B_{\rho}(x, r_n)\cap \cA) \leq \left(\int_{\R^N} f(x)^p dx\right)^{1/p}  \left(Leb^N(B_{\rho}(x, r_n)\cap \cA)\right)^{1/q}, 
\]
where $p = 4j+1$ and $q = (4j+1)/4j$.
Since $\int_{\R^N} f(x)^{4j+1} dx < \infty$, $\rho$ satisfies (P$2$) and $r_n \to 0$, it follows that in the thermodynamic regime $\Prob(X_1 \in B_{\rho}(x, r_n)\cap \cA)$ 
is bounded by $c_1 n^{-1/q}$ for some constant $c_1$.

For $m > n \ge j$, since each $j$-simplex in $\cC(\fX_m, r_n, \rho) \setminus \cC(\fX_n, r_n, \rho)$ must contain at least one vertex in $\{X_{n + 1}, \dots, X_m\}$, we have 
\[
	S_j(m,n) \le \sum_{i = n + 1}^m \xi(X_i, \fX_m),
\]
where $\xi(X_i, \fX_m)$ counts the number of $j$-simplices with one vertex $X_i$ in $\cC(\fX_m, r_n, \rho)$. It follows that 
\begin{equation}\label{4th-moment}
	\Ex[S_j(m,n)^4] \le (m-n)^3 \sum_{i = n+1}^m \Ex[\xi(X_i, \fX_m)^4] = (m-n)^4 \Ex[\xi(X_1, \fX_m)^4].
\end{equation}
Note that $\xi(X_1, \fX_m)$ can be written as follows
\begin{equation*}
	\xi(X_1, \fX_m) = \sum_{\{i_1, \dots, i_j\}} h_{j,r_n, \rho}(X_1, X_{i_i}, \dots, X_{i_j}) =:\sum_{{\bf i} = \{i_1, \dots, i_j\}} \eta_{\bf i},
\end{equation*}
where $\{i_1, \dots, i_j\}$ is a subset of $\{2,3, \ldots, m\}$.
We estimate the fourth moment of $\xi(X_1, \fX_m)$,
\begin{equation}\label{4th-moment-sum}
	\Ex[\xi(X_1, \fX_m)^4] = \sum_{\bf i, \bf j, \bf k, \bf l} \Ex[\eta_{\bf i} \eta_{\bf j} \eta_{\bf k} \eta_{\bf l}].
\end{equation}
Let $wt({\bf i, j, k, l}) =|\{{\bf i \cup j \cup k \cup l}\}|$. Then 
\[
	\Ex[\eta_{\bf i} \eta_{\bf j} \eta_{\bf k} \eta_{\bf l}] \le \Prob(X_i \in B_{\rho} (X_1, r_n) \cap \cA\colon i \in   {\bf i \cup j \cup k \cup l} ) \le   \left( \frac{c_1}{n^{1/q}} \right)^{wt({\bf i, j, k, l})}.
\]
Note that $j \le wt({\bf i, j, k, l})  \le 4j$ and given $wt({\bf i, j, k, l}) = w$, the number of terms in the sum~\eqref{4th-moment-sum} is bounded by  
\[
	\binom{m - 1}{w} \binom{w}{j}^4 \le c(w, j) m^w,
\]
where $c(w, j)$ is a constant depending on $w$ and $j$.
Therefore, the sum~\eqref{4th-moment-sum} is bounded by 
\[
	\Ex[\xi(X_1, \fX_m)^4] \le c \left( \frac{m}{n^{1/q}}\right)^{4j},
\]
where $c$ is a constant which  does not depend on $m$ and $n$. Combining the formula for the fourth moment~\eqref{4th-moment} and the above estimate, we have
\[
	\Ex[S_j(m,n)^4] \le c  (m - n)^4 \left( \frac{m}{n^{1/q}}\right)^{4j} .
\]

When $j \le m < n$, we change the role of $m$ and $n$ to get 
\[
	\Ex[S_j(m, n)^4] \le c (n - m)^4 n^{4j/p} =  c (n - m)^4 n^{\delta},
\]
where $\delta = 4j/(4j + 1) \in (0,1)$.
Combining two estimates, we have 
\[
	\Ex[S_j(m, n)^4] \le c  (m - n)^4 \left[n^\delta + \left( \frac{m}{n^{1/q}} \right)^{4j} \right].
\]
Therefore,
\begin{align*}
\Ex \left[\left|S_j( \cP_n, r_n, \rho) - S_j (\fX_n, r_n, \rho) \right|^4 \right]  &\le c \Ex \left[ (N_n - n)^4 \left( n^\delta + \frac{(N_n)^{4j}}{n^{4j/q}} \right) \right] \\
	&\le c \Ex[(N_n - n)^8]^{1/2} \Ex \bigg[ \left( n^\delta + \frac{(N_n)^{4j}}{n^{4j/q}} \right)^2 \bigg]^{1/2}.
\end{align*}
Here in the last inequality, we have used H\"{o}lder's inequality. Note that $\Ex[(N_n)^j]$ is a polynomial in $n$ of degree $j$. 
Thus the second factor in the above estimate is of $O(n^\delta )$. 
It is easy to check that $\Ex[(N_n - n)^8]$ is a polynomial of degree $4$ in $n$. Therefore, 
\[
        \Ex[(S_j( \cP_n, r_n, \rho) - S_j(\fX_n, r_n, \rho) )^4] = O(n^{2+\delta}).
\]
This completes the proof.
\end{proof}

\section{Betti numbers in the Euclidean setting}\label{betti_numbers}
We give the proof of Theorem~\ref{thm:euclidb} (for Poisson point processes) in this section. We first deal with the case where $\cA$ is compact and the function $f(x)$ is bounded. Theorem~\ref{thm:euclidb} then follows with the help of the strong law of large numbers for simplex counts and Lemma~\ref{lem:Betti-estimate}.

\subsection{Betti numbers in a compact region}\label{compact_region}
We prove the following law of large numbers. 

\begin{proposition}\label{pro:S,bdd}
Let $(\cA, \rho)$ be a metric space, where $\cA$ is a compact subset of $\R^N$ with $Leb^N(\partial \cA) = 0$ and the metric $\rho$ satisfies the property \emph{(P1)}.
Assume that $f(x)$ is a non negative function on $\cA$ and is bounded. 
Then as $n \to \infty$ with $n^{1/N}r_n \to r \in (0, \infty)$, 
\[
         \frac{\beta_k(\cP_n, r_n, \rho)}{n}  \to \int_{\cA} \hat {\beta}_k^{(N)} \left( \frac{f(x)}{D(x)}, r \right) D(x) dx ~ ~\text{\rm{a.s.}}
\]
Here, $\cP_n$ is a Poisson point process on $\cA$ with intensity function $n f(x)$.
\end{proposition}

Note that since $\cA$ is assumed to be compact, the metric $\rho$ also satisfies the property (P2).

To obtain Proposition~\ref{pro:S,bdd}, we partition the set $\cA$ as follows. Let $\alpha_n = r/r_n$. For fixed $L >0$ and for each $n$, 
divide $\R^N$ according to the lattice $ (L/\alpha_n^N)^{1/N} \Z^N$. Let $\{C_{n,i}\}$ be the cubes of the form $(a,b]^N$ that intersect with $\cA$, 
where $a,b \in \R$ and $b-a = (L/\alpha_n^N)^{1/N}$. Let $\cA_n = \cup_i C_{n,i}$. 
Observe that although $\rho$ is defined only on 
$C_{n,i} \cap \cA$, the object $\cC(\cP_n|_{C_{n,i}}, r_n, \rho)$ 
is well-defined since $\supp f \subset \cA$.
 The limiting behavior of $\beta_k(\cP_n, r_n, \rho)$ will be estimated by that of $\beta_k\left( \cup_{i}\cC(\cP_n|_{C_{n,i}}, r_n, \rho)\right)$, where 
the latter is studied in the following way.

 Consider the map $x \mapsto \alpha_n x$ and let $W_{n, i}$ be the image of $C_{n, i}$. Define a metric on $\alpha_n \cA$ as 
 \[
	\rho_n (x,y):= \alpha_n \rho \left(x/\alpha_n, y/ \alpha_n\right).
\] 
Let $\tilde \cP_n= \alpha_n\cP_n$. Then $\tilde \cP_n$ is 
a Poisson point process on $\R^N$ with intensity function 
\[
 n/\alpha_n^N f(x/\alpha_n) =: f_n(x).
\]
Note that for every
realization of $\cP_n$ on $\cA$,  $\cC(\cP_n|_{C_{n,i}} , r_n, \rho) \cong \cC(\tilde \cP_n|_{W_{n,i}} , r, \rho_n)$. It then follows that for fixed $L$ and $n$, 
\begin{equation}\label{betti_part}
     \beta_k\left( \bigcup_{i}\cC(\cP_n|_{C_{n,i}}, r_n, \rho)\right) = \sum_{i} \beta_k (\tilde \cP_n|_{W_{n,i}} , r, \rho_n),
\end{equation}
\begin{equation}\label{simp_part}
      S_j\left( \bigcup_{i}\cC(\cP_n|_{C_{n,i}}, r_n, \rho)\right) = \sum_{i} S_j (\tilde \cP_n|_{W_{n,i}} , r, \rho_n),
\end{equation}
because of the disjoint union of simplicial complexes. Observe that the above sums are of independent random variables, from which the following strong laws of large numbers for $ \beta_k\left( \cup_{i}\cC(\cP_n|_{C_{n,i}}, r_n, \rho)\right)$ and $ S_j\left( \cup_{i}\cC(\cP_n|_{C_{n,i}}, r_n, \rho)\right)$ hold.

\begin{lemma} \label{lem:SLLN for partition}
For fixed $L > 0$,  
as $n \to \infty$, 
 \begin{align*}
      &\text{\rm{(a)}}~~ \frac{1}{n} \sum_i \beta_k(\tilde \cP_n|_{W_{n,i}} , r, \rho_n) \to  \int_{\cA} \frac {\Ex [\beta_k(\cP_{L}({f}(x)), r, d_{x})]} { L}  dx  ~\text{\rm{a.s.}}, \\
       &\text{\rm{(b)}}~~  \frac{1}{n} \sum_i S_j(\tilde \cP_n|_{W_{n,i}} , r, \rho_n) \to \int_{\cA} \frac {\Ex [S_j(\cP_{L}({f}(x)), r, d_{x})]} { L}  dx  ~\text{\rm{a.s.}}
    \end{align*}
    Here, $\cP_L(\lambda)$ is the restriction on $W_L=(-L/2)^{1/N}, (L/2)^{1/N}]^N$ of a homogeneous Poisson point process $\cP(\lambda)$ on $\R^N$ with intensity $\lambda \ge 0$.
  \end{lemma}

 Now letting $L \to \infty$, we obtain the following result. 
Recall that $\hat S_j^{(N)} \left({\lambda}/{D(x)}, r\right) = A_j^{(N)}(r)(\lambda/D(x))^{j+1}$ from Corollary~\ref{corhom}.

\begin{lemma} \label{lem:Ltoinfty}
As $L \to \infty$,
\begin{align*}
    &\text{\rm{(a)}}~~ \int_{\cA} \frac {\Ex [\beta_k(\cP_{L}({f}(x)), r, d_{x})]} { L}  dx \to \int_{\cA} \hat \beta_k^{(N)} \left(\frac{f(x)}{D(x)}, r \right) D(x) dx,\\
    &\text{\rm{(b)}}~~ \int_{\cA} \frac {\Ex [S_j(\cP_{L}({f}(x)), r, d_{x})]} { L}  dx \to \int_{\cA} \hat S_j^{(N)} \left( \frac{f(x)}{D(x)}, r\right) D(x) dx. 
\end{align*}
\end{lemma}

The proof of Lemmas~\ref{lem:SLLN for partition} and \ref{lem:Ltoinfty} will be given later in this subsection. By using them, Proposition~\ref{pro:S,bdd} is shown as follows. 
\begin{proof}[Proof of Proposition~{\rm\ref{pro:S,bdd}}]
 For fixed $n$ and $L >0$, since the union $\cup_{i}\cC(\cP_n|_{C_{n,i}}, r_n, \rho)$ is a subcomplex of $\cC(\cP_n, r_n, \rho)$, it follows from Lemma~\ref{lem:Betti-estimate} and the equations~\eqref{betti_part}, \eqref{simp_part} that 
 \[
 \left| \beta_k(\cP_n, r_n, \rho) - \sum_{i} \beta_k(\tilde \cP_n|_{W_{n,i}}, r, \rho_n) \right| \leq \sum_{j=k}^{k+1} \left( S_j(\cP_n, r_n, \rho) - \sum_i S_j(\tilde \cP_n|_{W_{n,i}}, r, \rho_n) \right).
\]
Divide both sides by $n$. It follows from Proposition~\ref{thm: 4-order} and Lemma~\ref{lem:SLLN for partition} that in the thermodynamic regime, almost surely,
\begin{align*}
 \limsup_{n \to \infty}  \frac{\beta_k(\cP_n, r_n, \rho)}{n} \leq & \int_{\cA} \frac {\Ex [\beta_k(\cP_{L}({f}(x)), r, d_{x})]} { L}  dx\\
 &+ \sum_{j=k}^{k+1} \left| A_j^{(N)}(r) \int_{\cA} \frac{f(x)^{j+1}}{D(x)^j} dx -  \int_{\cA} \frac {\Ex [S_j(\cP_{L}({f}(x)), r, d_{x})]} { L}  dx \right|,
\end{align*}
\begin{align*}
\liminf_{n \to \infty}  \frac{\beta_k(\cP_n, r_n, \rho)}{n} \geq & \int_{\cA} \frac {\Ex [\beta_k(\cP_{L}({f}(x)), r, d_{x})]} { L}  dx\\
 &- \sum_{j=k}^{k+1} \left| A_j^{(N)}(r) \int_{\cA} \frac{f(x)^{j+1}}{D(x)^j} dx -  \int_{\cA} \frac {\Ex [S_j(\cP_{L}({f}(x)), r, d_{x})]} { L}  dx \right|.
 \end{align*}
 Now let $L \to \infty$. By Lemma~\ref{lem:Ltoinfty}, we obtain that almost surely,
\[
  \limsup_{n \to \infty}  \frac{\beta_k(\cP_n, r_n, \rho)}{n}  \leq  \int_{\cA} \hat \beta_k^{(N)} \left(\frac{f(x)}{D(x)}, r \right) D(x) dx,
\]
\[
 \liminf_{n \to \infty}  \frac{\beta_k(\cP_n, r_n, \rho)}{n}  \geq   \int_{\cA} \hat \beta_k^{(N)} \left(\frac{f(x)}{D(x)}, r \right) D(x) dx.
\] 
The proof is complete.
\end{proof}

Now what remains is to prove Lemmas~\ref{lem:SLLN for partition} and \ref{lem:Ltoinfty}.
To obtain the required results, the following 
implications of the coupling property of Poisson point processes, taken from \cite{Duy-2016, ysa}, are needed. Since $n/\alpha_n^N \to 1$ as $n \to \infty$ in the thermodynamic regime, choose $\Lambda > 0$ such that for all $n$ and $x \in W_{n,i}$,
$f_n(x) \leq \Lambda$.

\begin{enumerate}[(CP1)]
\item Let  $ t= \int_{W_{n,i}} f_n(x) dx = \int_{C_{n,i}} n f(x) dx$. Then the number of points in $W_{n,i}$, denoted by $N_t$, has a
Poisson distribution with parameter $t$. Clearly, $N_t$ is stochastically dominated by $N_{\Lambda L}$, which is a Poisson random variable 
with parameter $\Lambda L$. Therefore, 
\begin{equation}
\label{eqn : dominated}
\Ex[N_t ^{k}] \leq \Ex[N_{\Lambda L} ^{k}] \leq c(k, \Lambda L),
\end{equation}
where $c(k, \Lambda L)$ is a constant depending only on $k$ and $\Lambda L$. 

Also, for $\nu \in \N$,
\begin{align}
\label{eqn : tr bound}
       (\beta_k(\tilde \cP_n|_{W_{n,i}} ,r, \rho_n))^{\nu} \leq (S_k(\tilde\cP_n|_{W_{n,i}} , r, \rho_n))^{\nu} \leq N_t^{(k+1)\nu}.
\end{align}
Combining the relations \eqref{eqn : dominated} and  \eqref{eqn : tr bound} yields
\begin{equation*}\label{nuniformbound}
\Ex[ (\beta_k(\tilde \cP_n|_{W_{n,i}} ,r, \rho_n))^{\nu} ] \leq \Ex[ (S_k(\tilde \cP_n|_{W_{n,i}} ,r, \rho_n))^{\nu} ] \leq c(k, \nu, \Lambda L).
\end{equation*}
 \item{Let $r \in (0, \infty)$, $\lambda \geq 0$, and $A$ be a bounded Borel subset of $\R^N$.
 Let $S_j(\cP(\lambda), r; A)$ be the number of $j$-simplices that has at least one vertex in $A$. Then for $\lambda \leq \Lambda$,
\[
     \Ex[S_j(\cP(\lambda)|_A, r)] \leq \Ex[S_j(\cP(\Lambda), r; A)] \leq C(\Lambda, r, j)Leb^{N}(A),
 \]   
 where $C(\Lambda, r, j)$ is a constant depending only on $\Lambda, r$ and $j$.
 }
\end{enumerate}

We first discuss the proof of the statement (a) of Lemma~\ref{lem:SLLN for partition}, while its statement (b) follows similarly.
The almost sure convergence in Lemma~\ref{lem:SLLN for partition}, after establishing the convergence in expectation,  
follows from the following result:
\begin{lemma}
\label{lem:almost_sure}
Assume that for each $n$, the sequence $\{\xi_{n, i}\}_{i = 1}^{T_n}$ consists of independent random variables and that 
\[
	\sup_n\sup_{i} \Ex[|\xi_{n, i}|^4] < \infty.
\]
Assume further that $T_n/n \to \alpha \in (0, \infty)$ as $n \to \infty$. 
Then as $n \to \infty$,
\[
	\frac{1}{n} \sum_{i = 1}^{T_n}\Big( \xi_{n, i} - \Ex[\xi_{n, i}]\Big) \to 0~ \text{\rm{a.s.}}
\]
In addition, if  $\frac{1}{n}  \sum_{i = 1}^{T_n} \Ex[\xi_{n, i}] \to \mu$
then $\frac{1}{n}  \sum_{i = 1}^{T_n} \xi_{n, i} \to \mu \text{ almost surely as }n\to \infty$.
\end{lemma}
Lemma~\ref{lem:almost_sure} can be easily proved by calculating the order of the fourth moments.

By taking $\xi_{n, i}\ = \beta_k(\tilde \cP_n|_{W_{n,i}} , r, \rho_n)$ in Lemma \ref{lem:almost_sure}, 
their almost sure convergence follows from their convergence in expectation because 
by the property (CP1),
\[
       \sup_n\sup_{i} \Ex[(\beta_k(\tilde \cP_n|_{W_{n,i}} , r, \rho_n))^4] \leq  \sup_n\sup_{i} \Ex[(S_k(\tilde \cP_n|_{W_{n,i}} , r, \rho_n))^4]  < c(k, 4, \Lambda L) < \infty;
\]
\[ 
     \text{and since}~T_n = |\{C_{n,i}\}|,   \frac{T_n}{n} = \frac{Leb^{(N)}(\cA_n) \alpha_n^N}{n L} \to \frac{Leb^{(N)}(\cA)}{L} \in (0, \infty)~\text{as}~n \to \infty.
\]

Thus, the remaining thing in the proof of Lemma~\ref{lem:SLLN for partition} is to show the convergence in expectation. 
Let $\one_{C_{n,i}}$ be the indicator function of $C_{n,i}$.
The idea is to write the quantity $n^{-1} \sum_{i} \Ex[\beta_k(\tilde{\cP}_n|_{W_{n,i}}, r, \rho_n)]$ 
in terms of the integral of an appropriate function so that the Bounded Convergence Theorem (BCT) can be applied. 
For fixed $L > 0$, define the function $F_n\colon \R^N \to \R $ as 
\[
    F_n(x) :=  \frac{1} {L} \sum_i\Ex [\beta_k(\tilde \cP_n|_{W_{n,i}} , r, \rho_n)] \one_{C_{n,i}}(x).
\] 
Then
\begin{equation}\label{int_form}
 \int_{\cA_n} F_n (x) dx  = \frac{1}{\alpha_n^N} \sum_i\Ex [\beta_k(\tilde \cP_n|_{W_{n,i}} , r, \rho_n)].
\end{equation}
It is clear from the property~(CP1) that $F_n(x) \le L^{-1}c(k, 1, \Lambda L)$.

Now we consider the pointwise limit of $F_n$.  
If $x_0 \notin \cA$, there exists $n_{x_0} \in \N$ such that for all $n \geq n_{x_0}$, $ x_0 \notin \cA_n$. In this case, $F_n (x_0) \to 0$ as $n \to \infty$.
Let $x_0 \in \cA^\circ$ be a Lebesgue point of $f$ and $ \lambda  = f(x_0)$. In this case, 
the limiting behavior of $F_n(x_0)$ is determined by the following two estimates: 
 the non-homogeneous Poisson point process $\tilde \cP_n|_{W_{n,i_n}}$ is approximated by the homogeneous Poisson point process $\cP(\lambda)|_{W_{n, i_n}}$
and the non-homogeneous metric $\rho_n$ on $W_{n, i_n}$ is approximated by the homogeneous metric $d_{x_0}$ on $W_{n, i_n}$. Here $i_n$ is the unique index such that $x_0 \in C_{n, i_n}$.
The first estimate is as follows.

\begin{lemma} \label{lem:1_estimate} 
Let $x_0 \in \cA^\circ$ be a Lebesgue point of $f$ and $f(x_0) = \lambda$. Then as $n \to \infty$ with $n^{1/N}r_n \to r \in (0, \infty)$,
\begin{align*}
& \text{\rm{(a)}}~~ \big|\Ex[ \beta_k (\tilde \cP_n|_{W_{n,i_n}}, r, \rho_n)] - \Ex[ \beta_k (\cP(\lambda)|_{W_{n,i_n}}, r,  \rho_n)]\big| \to 0, \\
& \text{\rm{(b)}}~~\big|\Ex[ S_j (\tilde \cP_n|_{W_{n,i_n}}, r, \rho_n)] - \Ex[ S_j (\cP(\lambda)|_{W_{n,i_n}}, r, \rho_n)]\big| \to 0. 
\end{align*}
\end{lemma}
\begin{proof} (a)
 Let $\Phi $ be a homogeneous Poisson point process on $W_{n,i_n}\times [0, \infty )$ with intensity $1$. Define the following:
\[
     A_n = \{(x,t) \in W_{n,i_n} \times [0, \infty)\colon t \leq f_n(x) \}  ;~~  B_n =  \{(x,t) \in W_{n,i_n} \times [0, \infty)\colon t \leq \lambda \}.
\]
Then $\tilde \cP_n|_{W_{n,i_n}}$ (resp.~$\cP(\lambda)|_{W_{n,i_n}}$) has the same distribution with the projection of $\Phi |_ {A_n}$ (resp.~$\Phi |_ {B_n}$) onto $W_{n, i_n}$. 
Let 
\[
\Delta_n = \{(x,t) \in W_{n,i_n} \times [0, \infty)\colon  \min(f_n(x), \lambda) \leq t \leq   \max(f_n(x),\lambda) \}.
\] 
By using this coupling, $\tilde \cP_n|_{W_{n,i_n}} $ is identical with $\cP(\lambda)|_{W_{n,i_n}} $, denoted by $\tilde \cP_n|_{W_{n,i_n}} \equiv \cP(\lambda)|_{W_{n,i_n}}$,  if and only if there is no point of $\Phi $ in the region $\Delta_n$.
Thus
\begin{equation*} \label{prob_iden}
  \Prob(\tilde \cP_n|_{W_{n,i_n}} \equiv \cP(\lambda)|_{W_{n,i_n}}) = \exp \Big( - \int_{W_{n,i_n}} |f_n(x) - \lambda| dx \Big).
\end{equation*}
Consider
\begin{align*}
    \int_{W_{n,i_n}} |f_n(x) - \lambda| dx &= \int_{W_{n,i_n}} \left|\frac{n}{\alpha_n^N}  f\left(\frac{x}{\alpha_n}\right) -\lambda \right| dx\\
     &\leq \frac{n}{\alpha_n^N} \int_{W_{n,i_n}} \left|f \left(\frac{x}{\alpha_n}\right) -\lambda\right| dx +  \int_{W_{n,i_n}} \lambda \left|\frac{n}{\alpha_n^N} -1\right| dx \\
    & = \frac{nL}{\alpha_n^N Leb^{N}(C_{n,i_n})}  \int_{C_{n,i_n}} |f(x) - \lambda| dx + L\lambda \left|\frac{n}{\alpha_n^N} -1\right|.
\end{align*}
As $n \to \infty$, $n/\alpha_n^N \to 1$ and $\frac{1} {Leb^{N}(C_{n,i_n})}\int_{C_{n,i_n}} |f(x) - \lambda| dx \to 0$ because $x_0$ is a Lebesgue point of $f$. Therefore,
\[
      \Prob(\tilde \cP_n|_{W_{n,i_n}} \equiv \cP(\lambda)|_{W_{n,i_n}}) \to 1 ~ \text{as}~ n \to \infty.
\]
Since $\Prob (\beta_k (\tilde \cP_n|_{W_{n,i_n}}, r, \rho_n)  = \beta_k (\cP(\lambda)|_{W_{n,i_n}}, r,  \rho_n)) \geq \Prob(\tilde \cP_n|_{W_{n,i_n}} \equiv \cP(\lambda)|_{W_{n,i_n}})$, it follows that
\begin{equation*}\label{nice_coup_1}
 \beta_k (\tilde \cP_n|_{W_{n,i_n}}, r, \rho_n) - \beta_k (\cP(\lambda)|_{W_{n,i_n}}, r, \rho_n)  \to 0 \text{ in probability}.
\end{equation*}
In addition, we have
\[
\Ex[(\beta_k (\tilde \cP_n|_{W_{n,i_n}}, r, \rho_n) - \beta_k (\cP(\lambda)|_{W_{n,i_n}}, r, \rho_n))^2] \leq 2c(k, 2, \Lambda L ).
\] 
Therefore, the convergence in expectation, or the statement (a), follows by the corollary following Theorem~$25.12$ in \cite{billing}. The statement (b) is similarly proved. 
The proof is complete.
\end{proof}

For the second estimate, we need the following lemma and the uniform convergence of the metrics. Let $C_k(\cP_L(\lambda), r, d_x)$ be the set of $k$-simplices in $\cC(\cP_L(\lambda), r, d_x)$, where $x \in \cA$. 
 Note that if we define $C_k(\cP_L(\lambda), r^+, d_x):= \bigcap_{s>r} C_k(\cP_L(\lambda), s, d_x)$ then it follows from the definition of \v{C}ech complexes that
\[
C_k(\cP_L(\lambda), r^+, d_x)= C_k(\cP_L(\lambda), r, d_x).
\]
As a consequence, $\Ex[S_k(\cP_L(\lambda), r, d_x)]$ is right continuous at $r$. Moreover, by the scaling property of Poisson point processes and Theorem~\ref{palm1}, it can be easily shown that $\Ex[S_k(\cP_L(\lambda), r, d_x)]$ is also left continuous at $r$.

\begin{lemma}\label{continuity_simp}
Let 
$ C_k(\cP_L(\lambda), r^-, d_x) : = \bigcup_{s < r} C_k(\cP_L(\lambda), s, d_x)$.
Then for fixed $r > 0$,
\[
           \Prob (C_k(\cP_L(\lambda), r, d_x)=C_k(\cP_L(\lambda), r^-, d_x)) = 1,
 \]
where `$=$' means set equality. 
\end{lemma}
\begin{proof}
Let $\{s_n\}$ be an increasing sequence of positive real numbers converging to $r$ as $n \to \infty$. Since for all $n$, the set 
$ C_k( \cP_L(\lambda), s_n, d_x)$ is a subset of  $C_k(\cP_L(\lambda), r, d_x) $, the almost sure convergence of 
 $S_k( \cP_L(\lambda), s_n, d_x)$ to $S_k( \cP_L(\lambda), r, d_x)$, as $s_n$ approaches $r$, implies the desired result.
 
Note that for any realization of $\cP(\lambda)$ in $W_L$, 
$S_k( \cP_L(\lambda), s_n, d_x)$ is an increasing sequence, which is  bounded by $S_k( \cP_L(\lambda), r, d_x)$ and  
hence converges to a finite value, denoted by $F(\cP_L(\lambda), r)$. 
This implies as $s_n \to r$,
\[
 \Ex[S_k(\cP_L(\lambda), s_n, d_x)] \to \Ex[F(\cP_L(\lambda), r)].
\]  
But since $\Ex[S_k(\cP_L(\lambda), r, d_x)]$ is left continuous at $r$, it follows that $F(\cP_L(\lambda), r) = S_k(\cP_L(\lambda), r, d_x)$ almost surely. 
Thus as $s_n \to r$, 
\[
S_k( \cP_L(\lambda), s_n, d_x) \to S_k( \cP_L(\lambda), r, d_x) ~\text{a.s.},
\]
which completes the proof.
 \end{proof}

In the following definition and lemma, the metrics $\rho_n$ and $ d_{x_0}$ on $W_{n,i_n}$ are abstract ones, although denoted by the same notation as used 
for defining particular metrics on $W_{n,i_n}$.
\begin{definition}
Let $\rho_n$ be a sequence of metrics on $W_{n,i_n}$. We say $\rho_n$ converges uniformly to a metric $d_{x_0}$ on $W_{n,i_n}$, if for given $\varepsilon > 0$ there 
exists $n_{x_0}$ such that for all $n \geq n_{x_0}$, whenever $x,y \in W_{n,i_n}$,
\[
        |\rho_n (x,y) - d_{x_0}(x,y)| \leq \varepsilon.
\]
\end{definition}

It is clear from~(P$1$) and the inequality \eqref{prop_metric_d} that $\rho_n (x,y) = \alpha_n \rho \left(x/\alpha_n, y/ \alpha_n\right)$ converges uniformly to 
$d_{x_0}(x,y)= \|\B_{x_0}(x-y)\|$ on $W_{n,i_n}$, where $x_0 \in \cA^\circ$.
Now we state and prove the second estimate.

\begin{lemma} \label{lem:2_estimate}
Let $r \in (0, \infty)$ and $\lambda \in [0, \infty)$.
Assume that $ \rho_n$ be a sequence of metrics on $U_n$ converging uniformly to a metric $d_{x_0}$ on $U_n$. Assume further that for each $n$, $W_{n, i_n}$ is a cube of length $L$ satisfying that $\cup_{y \in W_{n, i_n}} B_{\rho_n}(y, r) \subset U_n$. Then as $n \to \infty$,
\begin{align*}
&\text{\rm{(a)}}~~ | \Ex[ \beta_k (\cP(\lambda)|_{W_{n,i_n}}, r,  \rho_n)] - \Ex[ \beta_k (\cP(\lambda)|_{W_{n,i_n}}, r, d_{x_0})]| \to 0,  \\
&\text{\rm{(b)}}~~  | \Ex[ S_j (\cP(\lambda)|_{W_{n,i_n}}, r,  \rho_n)] - \Ex[ S_j (\cP(\lambda)|_{W_{n,i_n}}, r, d_{x_0})]| \to 0. 
\end{align*} 
\end{lemma}
\begin{proof} (a) 
The statement hold trivially for $\lambda =0$. Let $\lambda > 0$. Let $W_{n,i_n} = y_n + W$, where $W$ is a fixed window of volume $L$. Define a metric on $U_n - y_n$ as 
\[
\tilde \rho_n(x,y) :=  \rho_n (x+y_n, y+y_n).
\]
Then  $\beta_k(\cP(\lambda)|_{W_{n,i_n}}, r, \rho_n) $ has the same distribution with $\beta_k(\cP_L(\lambda), r, \tilde\rho_n)$. Note that the metric $d_{x_0}$ is translation invariant. Thus, $\beta_k(\cP(\lambda)|_{W_{n,i_n}}, r, d_{x_0}) $ also has the same distribution with $\beta_k(\cP_L(\lambda), r, d_{x_0})$. Therefore, it is clear from the proof of Lemma~\ref{lem:1_estimate} that it is sufficient to show here that almost surely,  
\begin{equation}\label{coincide}
\cC(\cP_L(\lambda), r, \tilde \rho_n) = \cC(\cP_L(\lambda), r, d_{x_0}), \text{ for $n$ large enough.} 
\end{equation}

By Lemma~\ref{continuity_simp}, almost surely,
\[
	\cC(\cP_L(\lambda), r^-, d_{x_0}) = \bigcup_{s< r}  \cC(\cP_L(\lambda), s, d_{x_0}) = \cC(\cP_L(\lambda), r, d_{x_0}).  
\]
Fix  a configuration such that the above holds. The proof is complete by showing  the identity~\eqref{coincide}.

 Let $\sigma = \{v_0, v_1, \ldots, v_k\} $ be a $k$-simplex in $\cC(\cP_L(\lambda), r, d_{x_0})$. We first show that $\sigma$ belongs to $\cC(\cP_L(\lambda), r, \tilde \rho_n)$ when $n$ is large enough. Indeed, $\sigma \in \cC(\cP_L(\lambda), s, d_{x_0})$ for some $s < r$. Take a point $v \in \bigcap_{i=0}^k B_{d_{x_0}}(v_i, s)$. Then $d_{x_0}(v, v_i) \le s$, for all $ 0\leq i \leq k$. Since  $\tilde \rho_n$ converges uniformly to $d_{x_0}$, there exists $n_{x_0} \in \N$ such that for all $n \geq n_{x_0}$ and for all $ 0\leq i \leq k$,
\[
            \tilde\rho_n(v, v_i) \leq (r - s)+ d_{x_0}(v, v_i) \leq r,
\]
which implies that $\sigma \in \cC(\cP_L(\lambda), r, \tilde \rho_n)$.

Conversely, we now show that, if $\sigma \in \cC(\cP_L(\lambda), r, \tilde \rho_{n_j})$, for some subsequence $\{n_j\}$ tending to infinity, then $\sigma \in \cC(\cP_L(\lambda), r, d_{x_0})$.  Let $v_{n_j} \in \cap_{i=0}^{k} B_{\tilde \rho_{n_j}}(v_i, r)$. Because of the compactness,  
there exists a subsequence $\{n_j'\}$ of $\{n_j\}$ such that $\{v_{n'_j}\}$ converges to  $v_\infty$.
 By the triangle inequality, 
\[
    d_{x_0}(v_\infty, v_i) \leq d_{x_0}(v_\infty, v_{n'_j}) + d_{x_0}(v_{n'_j}, v_i).
\]
Clearly, $d_{x_0}(v_\infty, v_{n'_j}) \to 0$ as $n \to \infty$. In addition, by the uniform convergence assumption, for given $\varepsilon > 0$, there exists $n_{x_0} \in \N$ such that for all $n'_j \geq n_{x_0}$ and for all $ 0\leq i \leq k$,
\[
         d_{x_0}( v_{n'_j}, v_i) \leq   r + \varepsilon.
\]
Since $\varepsilon$ is arbitrary, this implies $v_{\infty} \in \cap_{i=0}^{k}  B_{d_{x_0}}(v_i, r)$, and hence, $\sigma \in \cC(\cP_L(\lambda), r, d_{x_0})$. The lemma is proved.
\end{proof}

 Finally we have all the ingredients to calculate the point-wise convergence of $F_n(x_0)$. Recall 
 $x_0 \in \cA^\circ$ is the Lebesgue point of $f$ and $ \lambda  = f(x_0)$.
 Since the metric $d_{x_0}$ is translation invariant and $\cP(\lambda)$ is stationary, we have 
\[
\Ex[ \beta_k (\cP(\lambda)|_{W_{n,i_n}}, r, d_{x_0})] = \Ex[ \beta_k (\cP_L(\lambda), r, d_{x_0})]. 
\]
 So by  Lemmas~\ref{lem:1_estimate} and \ref{lem:2_estimate}, and by the Lebesgue differentation theorem, for almost everywhere $x_0$, as $n \to \infty$, 
\[
       F_n(x_0)  \to \frac{\Ex[ \beta_k (\cP_L(\lambda), r, d_{x_0})]}{L}.
\]
Therefore by BCT, as $n \to \infty$,
 \[
       \int_{\cA_n} F_n (x) dx  \to \int_{\cA} \frac {\Ex [\beta_k(\cP_{L}({f}(x)), r, d_{x})]} { L}  dx.
 \]
 Substitution of the value of the integral of $F_n$ from \eqref{int_form} in the above expression yields as $n \to \infty$, 
 \[
       \frac{1}{n} \sum_i\Ex [\beta_k(\tilde \cP_n|_{W_{n,i}} , r, \rho_n)] \to \int_{\cA} \frac {\Ex [\beta_k(\cP_{L}({f}(x)), r, d_{x})]} { L}  dx.
 \]
This completes the proof of the statement (a) of Lemma \ref{lem:SLLN for partition}. Since Lemmas~\ref{lem:1_estimate} and \ref{lem:2_estimate} also hold for $S_j(\cdot)$, the statement (b) follows similarly.

Let us conclude this subsection with the proof of Lemma \ref{lem:Ltoinfty}.
\begin{proof}[{Proof of Lemma {\rm\ref{lem:Ltoinfty}}}]
Let $x_0 \in \cA^\circ$ and $f(x_0) = \lambda$.  
By the inequality~\eqref{prop_metric_d}, 
\[ 
\cC(\cP_{L}(\lambda), r, d_{x_0}) \subset \cC(\cP_{L}(\lambda), r/c),
\]
where $c$ is the constant in \eqref{prop_metric_d}.
Thus, 
\[
         \Ex [\beta_k(\cP_{L}(\lambda), r, d_{x_0})] \leq   \Ex [S_k(\cP_{L}(\lambda), r, d_{x_0})] \leq \Ex [S_k(\cP_{L}(\lambda), r/c)].
\]
Therefore $ L^{-1}\Ex [\beta_k(\cP_{L}(\lambda), r, d_{x_0})] $ and $L^{-1}\Ex [S_j(\cP_{L}(\lambda), r, d_{x_0})]$   are uniformly bounded in $L$ by the property~(CP$2$).

Recall from Corollary~\ref{corhom} that, as $L \to \infty$,
\[
\frac{\Ex [S_j(\cP_{L}(\lambda), r, d_{x_0})]}{L} \to  \hat S_j^{(N)} \left(\frac{\lambda}{D(x_0)}, r\right) D(x_0),
\]
and from Theorem~$1.5$ in \cite{Duy-2016}, as $L \to \infty$,
\[
\frac {\Ex [\beta_k(\cP_{L}(\lambda), r, d_{x_0})]} {L}  = \frac {\Ex [\beta_k(\cP_{\tilde L}(\lambda/D(x_0)), r)]} {\tilde L} D(x_0) \to  \hat \beta_k \left(\frac{ \lambda}{D(x_0)}, r\right) D(x_0),
\]
where $\tilde L= L D(x_0)= Leb^{N}(\{\B_{x_0} x\colon x\in W_{L} \})$.
Hence, by BCT, we obtain the desired result.
\end{proof}

\subsection{Betti numbers in Euclidean spaces}
This subsection contains the proof of  Theorem~\ref{thm:euclidb}. 
\begin{proof}[{Proof of Theorem {\rm\ref{thm:euclidb}}} (for Poisson point processes)]
Since $Leb^N(\cA) = 0$, we assume $\cA$ is closed subset of $\R^N$ without loss of generality. 
Let $\cA_i$ be an increasing sequence of compact subsets of $\cA$ such that $\cup_i \cA_i = \cA$.
For each $i\geq 1$, define the function $ f_i\colon \R^N \to \R^+$ as 
    \[  f_i (x) =  \left\{ 
\begin{array}{ll}
       \min(f(x), i), & x\in \cA_i \\
     0, & \text{otherwise}. \\
\end{array} 
\right. \]
Let $\cP_{n}^{(i)}$ be the non-homogeneous Poisson point process on $\cA_i$ with intensity function $nf_i$. 
For each $i$, we have the following coupling 
\[
      \cP_{n} \dist \cP_{n}^{(i)} + \cP(n g_i), 
\]
where, for $x \in \R^N$, $g_i(x) = f(x) - f_i(x)$.
So by using Lemma~\ref{lem:Betti-estimate} for each $i$, we have
\[
    \left| \frac{\beta_k(\cP_{n}, r_n, \rho)}{n} - \frac{\beta_k(\cP_{n}^{(i)}, r_n, \rho)}{n}  \right| 
    \leq \sum_{j=k}^{k+1}\left( \frac{S_j(\cP_{n}, r_n, \rho)}{n} -\frac{S_j(\cP_{n}^{(i)}, r_n, \rho)}{n}\right).
\]
By first letting $n \to \infty$, and using Proposition \ref{pro:S,bdd} (since $f_i$ is bounded) and  Proposition \ref{thm: 4-order}, and then letting 
$i \to \infty$, we obtain that almost surely
\[
     \limsup_{n \to \infty}  \frac{\beta_k(\cP_n, r_n, \rho)}{n} \leq \lim_{i \to \infty} \int_{\cA_i} \hat {\beta}_k^{(N)} \left( \frac{f_i(x)}{D(x)}, r \right) D(x) dx, 
\]
\[
    \liminf_{n \to \infty}  \frac{\beta_k(\cP_n, r_n, \rho)}{n} \geq  \lim_{i \to \infty} \int_{\cA_i} \hat {\beta}_k^{(N)} \left( \frac{f_i(x)}{D(x)}, r \right) D(x) dx.
\]

Since $\hat {\beta}_k^{(N)}(0, r) = 0$ and $f_i(x) = 0$ for all $x \notin \cA_i$, the remaining task is to show that as $i \to \infty$, 
\begin{equation}\label{lim_hat_betti}
 \int_{\R^N} \hat {\beta}_k^{(N)} \left( \frac{f_i(x)}{D(x)}, r \right) D(x) dx \to \int_{\R^N} \hat {\beta}_k^{(N)} \left( \frac{f(x)}{D(x)}, r \right) D(x) dx.
\end{equation}
The limiting behavior \eqref{lim_hat_betti} can be easily obtained using DCT. Indeed, since 
the limiting constant $\hat \beta_k^{(N)}$ is continuous in the first parameter and $f_i(x) \to f(x)$ as $i \to \infty$, the point-wise limit holds.
For the uniform bound, since the $k$th Betti number is bounded from above by number of $k$-simplices, it follows from 
the definition of $\hat \beta_k$ and Corollary \ref{corhom} that for each $i$,
\[
       \hat \beta_k^{(N)}\left( \frac{f_i(x)}{D(x)}, r\right) \leq \hat S_k^{(N)}\left( \frac{f_i(x)}{D(x)}, r\right) \leq  \hat S_k^{(N)}\left( \frac{f(x)}{D(x)}, r\right),
\]
where the second inequality comes from the monotonicity of $\hat S_k^{(N)} (\lambda, r)$.  
The rightmost function in the above expression is integrable. This completes the proof.
\end{proof}

\begin{remark}\label{for mani}
It is clear that under the same setting as in Theorem~\ref{thm:euclidb}, as $n \to \infty$,
 \[
        \frac{\Ex[\beta_k(\fX_n, r_n, \rho)]}{n}  \to \int_{\R^N} \hat {\beta}_k^{(N)} \left( \frac{f(x)}{D(x)}, r \right) D(x) dx.
\]

 \end{remark}

\section{Manifold setting}\label{simplex-manifold}
Let us begin with the theory of integration of measurable functions on manifolds. This is an easy generalization of the integration of continuous functions considered in \cite{munkfold}.

Let $\cM \subset \R^N$ be a compact $m$-dimensional $C^1$ manifold and $\kappa$ be a measurable function on $\cM$.  Consider the case where the support of $\kappa$ can be covered by a single chart. Let $(V, \phi)$ be that chart. Since the support of $\kappa$ is 
compact and $\phi^{-1}$ is continuous, without loss of generality, we can assume that $V$ is bounded. Then the integral of $\kappa$ over $\cM$ is defined as
\[
      \int_{\cM} \kappa(z) dz = \int_{V^\circ} \kappa(\phi(x)) D_{\phi}(x) dx,
\]
provided that the right hand side is integrable.
Here, $V^\circ = V$ if $V$ is open in $\R^m$, otherwise $V^\circ = V \cap \bH^m_{+}$, where $\bH^m_{+}$ consists of $x \in \R^m$ for which $x_m > 0$, and 
$ D_{\phi}(x) = \det((\J_{\phi}^t (x)\J_{\phi} (x))^{1/2})$, where $\J_{\phi}^t (x)$  is the transpose of the Jacobian of $\phi$ at $x$.
The above integral is well-defined in the sense that it is independent of the choice of chart.

To define the integration in general, we need a concept of partition of unity \cite{munkfold}. 
Instead of mentioning it, we shall use the following lemma 
to carry out the integration of a measurable function $\kappa$ over $\cM$.

\begin{definition}
A subset $K$ of $\cM$ is said to have \textit{measure zero} in $\cM$ if it can be covered by countably many charts $\phi_i\colon V_i \to M$ such that the 
set 
\[
     K_i = \phi_i^{-1}(K \cap V_i)
\] 
has measure zero in $\R^m$ for each $i$.
\end{definition}

\begin{lemma}\label{integral_manifolds}
Suppose $(V_i, \phi_i)$, for $i = 1, 2, \ldots, l$, is a chart on $\cM$, such that $V_i$ is open in $\R^m$ and $\cM$ is the disjoint union of open sets
$\phi_1(V_1), \phi_2(V_2), \ldots, \phi_l(V_l)$
of $\cM$ and a set $K$ of measure zero in $\cM$. Then
\[
\int_\cM \kappa(z) dz = \sum_{i =1}^l \int_{V_i} \kappa(\phi_i(x)) D_{\phi_i} (x) dx.
\]
\end{lemma}

The set $\{(V_i, \phi_i): i=1,2, \ldots, l\}$ is also called an atlas for $\cM$. 
Lemma~\ref{integral_manifolds} gives us a way to calculate the integral of a measurable function over $\cM$.

To obtain results either for simplex counts or for Betti numbers in the manifold setting, we partition the manifold $\cM$ as follows.
Let $(V_y, \phi_y)$  be a chart for each $y \in \cM$. We may assume that $V_y$ is a  ball in $\R^m$ (or in $\bH^m$).
Choose $\delta_y > 0$ such that $B(y, \delta_y) \cap \cM \subset \phi_y(V_y)$ and $\partial (B(y, \delta_y/2) \cap \cM)$ has measure zero in $\cM$. Here the boundary is taken with respect to the topology on $\cM$.
Since $\cM$ is compact, we can find a finite index set $I$ such that
\[
       \cM = \bigcup_{i \in I}\left(B(y_i, \delta_{y_i}/2) \cap \cM\right).
\]
To simplify the notation, we replace $y_i$ by $i$ in the subscripts. Let $U_i = B(y_i, \delta_{y_i}/2) \cap \cM$.
Taking $M_1 = U_1$ and $M_i = U_i \backslash (\cup_{j=1}^{i-1} U_j$), we get the partition of $\cM$. 
Moreover, $\cM$ is the disjoint union of $\{M_i^\circ, i \in I\}$ and $\cup_{i\in I} \partial M_i$ whose measure is zero in $\cM$. 
Let $C_i = \phi_i^{-1}(M_i^\circ)$. 
 Let for each $i \in I$, $\rho_i(x, x') := \|\phi_i(x) -\phi_i(x')\|$ be the metric on $V_i=V_{y_i}$.

With the above partition of $\cM$, we show the following results for simplex counts in the manifold setting.
Note that the law of large numbers simplex counts in this setting may be proved by using arguments as in \cite{Penrose_mani}. 
However, to make the article self-contained, we discuss a proof here.

\begin{lemma}\label{simp_count}
Assume that $\int_{\cM} \kappa(z)^{j+1} dz < \infty$, and $\lim_{n \to \infty} r_n = 0$. Then 
\[
      \lim_{n \to \infty} r_n^{-mj} n^{-(j+1)} \Ex[S_j(\cQ_n, r_n)] = A_j^{(m)}(1) \int_{\cM} \kappa(z)^{j+1} dz.
\]
\end{lemma}

\begin{proof}
Similar to the proof of Proposition~\ref{general_simplicial}, we can write $r_n^{-mj} n^{-(j+1)} \Ex[S_j(\cQ_n, r_n)]$ as a sum of two terms $I_1^n$ and $I_2^n$. Moreover, using Lemma~\ref{integral_manifolds},  
\begin{align*}
    I_1^n &=  \frac{r_n^{-mj}}{(j+1)!}  \sum_{i \in I} \int_{M_i^\circ} \left( \int _{\cM^j}  h_{j, r_{n}} (z_0, \z) d\z \right)\kappa(z_0)^{j+1} dz_0,\\
      I_2^n  &= \frac{r_n^{-mj}}{(j+1)!} \sum_{i \in I} \int_{M_i^\circ}\left(\int_{\cM^{j}} h_{j, r_{n}} (z_0, \z) \left( \prod_{l=1}^j \kappa(z_l) -  \kappa(z_0)^{j} \right)d\z \right)\kappa(z_0) dz_0.
\end{align*}
Let $n' \in \N$ such that for all $n \geq n'$,  $r_n  \leq \min_{i \in I} (\delta_i/2)$. Then when $z_0 \in M_i^\circ \subset U_i$, the indicator function $h_{j,r_n}(z_0, \z)$ is equal to $1$ only if $z_1, \ldots, z_j \in  B(y_i, \delta_i) \cap \cM\subset \phi(V_i)$. Therefore, by the definition of integration in case of single chart, we obtain that, for all $n \geq n'$, 
\begin{align*}
    I_1^n =  \frac{r_n^{-mj}}{(j+1)!}  &\sum_{i \in I} \int_{V_i} \left( \int _{V_i^j}  h_{j, r_{n}, \rho_i} (x_0, \x) \prod_{l=1}^j D_{\phi_i}(x_i)d\x \right)\\
     &\times \kappa(\phi(x_0))^{j+1} D_{\phi_i}(x_0) \one_{C_i}(x_0) dx_0,
\end{align*}
\begin{align*}
  I_2^n = \frac{r_n^{-mj}}{(j+1)!} &\sum_{i \in I} \int_{V_i}\left(\int_{V_i^{j}} h_{j, r_{n}, \rho_i} (x_0, \x) \left( \prod_{l=1}^j \kappa(\phi_i(x_l)) -  \kappa(\phi_i(x_0))^{j} \right)
  \prod_{l=1}^j D_{\phi_i}(x_i) d\x \right)\\ 
  &\times \kappa(\phi(x_0))^{j+1} D_{\phi_i}(x_0) \one_{C_i}(x_0) dx_0,
\end{align*}
where $\one_{C_i}(x_0)$ is the indicator function of $C_i$.

Now all the arguments used in the proofs of Lemmas \ref{lem: ex_conv_sim} and \ref{lem: ex_conv_sim2} can be applied to show that $\lim_{n \to \infty} I_1^n = A_j^{(m)}(1) \int_{\cM} \kappa(z)^{j+1} dz$
 and  $\lim_{n \to \infty} I_2^n = 0$ respectively. 
This is because  for each $i$, the metric $\rho_i$ satisfies the properties (P$1$) and (P$2$) as we have shown in Section~\ref{motivation}, and the function $D_{\phi_i}(\cdot)$ is continuous on the compact subset $\bar C_i \subset \R^m$. 
\end{proof}

In the thermodynamic regime, the above lemma can be restated as follows:
\begin{corollary}\label{mani_simp}
 Assume that   $\int_{\cM} \kappa(z)^{j+1} dz < +\infty$ and $  \lim _{n \to \infty} n^{1/m} r_n = r \in (0, \infty)$. Then 
\[
   \lim_{n \to \infty} \frac{\Ex[S_j(\cQ_n, r_n)]}{n} =  A_j^{(m)} (r) \int_{\cM} \kappa(z)^{j+1} dz.
\]
\end{corollary}

Note that the proof of the strong law for simplex counts in the manifold setting is similar to that of Proposition~\ref{thm: 4-order}. 
So we state it without proof. 
Assume that $\int_{\cM} \kappa(z)^{4j+1} dz < +\infty$ and $  \lim _{n \to \infty} n^{1/m} r_n = r \in (0, \infty)$. Then 
as $n \to \infty$,
\begin{equation}\label{almost_mani}
\frac{S_j(\cQ_n, r_n)}{n} \to A_j^{(m)} (r) \int_{\cM} \kappa(z)^{j+1} dz~\text{a.s.}
\end{equation}

We have now all the required results to prove Theorem~\ref{thm:manifoldb} (for Poisson point processes). Before presenting its proof, we give the proof of the statement 
(a) of Lemma~\ref{conversion_poisson_binomial}.
\begin{proof}[{Proof of Lemma~{\rm\ref{conversion_poisson_binomial}} \rm{(a)}}]
From the proof of the statement (b) of this lemma, it is enough to show here that the volume of $B(z_0, r)\cap \cM$ is bounded by $c r^m$, for $r \le r_0$, where $r_0$ and $c$ are constants. Assume that $z_0 \in M_i$. Let $x_0 = \phi^{-1}(z_0)$. Then for $r \le  \min_{i \in I} \delta_i/2$, 
\[
	\int_{B(z_0, r)\cap \cM} dz = \int_{B_{\rho_i}(x_0, r)} D_{\phi_i}(x)dx.
\]
The required property is now trivial because the metric $\rho_i$ satisfies the property (P2) and the set $I$ is finite. The proof is complete.
\end{proof}

We conclude this article with the proof of Theorem~\ref{thm:manifoldb}.
 \begin{proof}[{Proof of Theorem~{\rm\ref{thm:manifoldb}}} (for Poisson point processes)]
Consider the limiting behavior of $\beta_k(\cQ_n|_{M_i^\circ}, r_n)$. 
 Define 
\[
 	f_i(x) = \kappa(\phi_i(x)) D_{\phi_i}(x) \one_{C_i}(x).
\]
Since $D_{\phi_i}(x)$ is a continuous function on the compact set $\bar C_i$, 
it follows from the assumption on $\kappa$ that $\int_{V_i} f_i(x)^j dx < +\infty$ for all $j \in \N$. Let $\cP_n = \phi^{-1}(\cQ_n|_{M_i^\circ})$. Then $\cP_n$  becomes a Poisson point process on $V_i$ with intensity function $nf_i$. It is clear that  
  \[
              \beta_k(\cQ_n|_{M_i^\circ}, r_n) = \beta_k(\cP_n, r_n, \rho_i),
  \] 
where recall that  $\rho_i(x,y) = \|\phi_i(x) - \phi_i(y)\|$, $x,y \in V_i$. Thus, by Theorem~\ref{thm:euclidb} for Poisson point processes, as $n \to \infty$ with $n^{1/m} r_n \to r$,
\[
     \frac{\beta_k(\cQ_n|_{M_i^\circ}, r_n)}{n} \to \int_{V_i} \hat \beta_k^{(m)} \left( \frac{f(x)}{D_{\phi_i}(x)}, r \right) D_{\phi_i}(x) dx = \int_{M_i^\circ} \hat \beta_k^{(m)}\left( \kappa(z), r \right)  dz \text{ a.s.}\\
\]
    
By Lemma \ref{lem:Betti-estimate}, 
\[
   \left| \frac{\beta_k(\cQ_n, r_n)}{n}  -  \sum_{i \in I} \frac{\beta_k(\cQ_n|_{M_i^\circ}, r_n)}{n} \right| \leq \sum_{j=k}^{k+1} \left(\frac{S_j(\cQ_n, r_n)}{n} - \sum_{i\in I}  \frac{S_j(\cQ_n|_{M_i^\circ}, r_n)} {n}\right).
\]
Let $n \to \infty$ with $n^{1/m} r_n \to r$. Then from the strong law \eqref{almost_mani}, the left hand side of the above inequality converges to  zero almost surely. Moreover, 
since $I$ is finite,
\begin{equation*}\label{mani_eq2}
       \sum_{i\in I} \frac{\beta_k(\cQ_n|_{M_i^\circ}, r_n)}{n}  \to  \sum_{i\in I} \int_{M_i^\circ} \hat \beta_k^{(m)}(\kappa(z), r) dz =\int_{\cM} \hat \beta_k^{(m)}(\kappa(z), r) dz ~\text{a.s.}
\end{equation*}
This completes the proof of Theorem~\ref{thm:manifoldb}.
\end{proof}

\section*{Acknowledgements}
The authors are thankful to Prof.~Tomoyuki Shirai for many useful discussions. This work is 
partially supported by JST CREST Mathematics (15656429). 
A.G.~is fully supported by JICA-Friendship Scholarship. K.D.T.~is partially supported by JSPS KAKENHI Grant Numbers JP16K17616.     
 K.T.~is partially supported by JSPS KAKENHI Grant Numbers 18K13426.

\begin{footnotesize}

\end{footnotesize}
\Addresses
\end{document}